\documentclass[11pt,reqno]{amsart}
\usepackage[margin=2.6cm]{geometry}
\usepackage{amsmath,amssymb,amsthm,mathrsfs, mathtools}
\usepackage{booktabs,array,longtable}
\usepackage{microtype}
\usepackage{xcolor}
\usepackage{hyperref}
\usepackage{enumitem}
\usepackage[nameinlink,capitalise]{cleveref}
\usepackage{url}
\usepackage{tabularx}
\usepackage{ragged2e}

\setlist[enumerate,1]{label=\roman*., font=\normalfont}
\hypersetup{colorlinks=true,linkcolor=blue!55!black,citecolor=blue!55!black,
        urlcolor=blue!55!black}
\numberwithin{equation}{section}

\newtheorem{theorem}{Theorem}[section]
\newtheorem{proposition}[theorem]{Proposition}
\newtheorem{lemma}[theorem]{Lemma}
\newtheorem{corollary}[theorem]{Corollary}

\newtheorem{conjecture}[theorem]{Conjecture}
\newtheorem{question}[theorem]{Question}
\theoremstyle{definition}
\newtheorem{remark}[theorem]{Remark}
\newtheorem{construction}[theorem]{Construction}

\newtheorem{example}[theorem]{Example}

\newcommand{\PP}{\mathbb{P}}
\newcommand{\CC}{\mathbb{C}}
\newcommand{\ZZ}{\mathbb{Z}}
\newcommand{\QQ}{\mathbb{Q}}
\newcommand{\OO}{\mathcal{O}}
\newcommand{\cO}{\mathcal O}
\newcommand{\Est}{E_{\mathrm{st}}}
\newcommand{\hst}{h_{\mathrm{st}}}
\newcommand{\hprim}{h_{\mathrm{prim}}}
\newcommand{\estr}{e_{\mathrm{st}}}
\newcommand{\indst}{\operatorname{ind}_{\mathrm{st}}}
\newcommand{\E}{E}
\newcommand{\Bl}{\operatorname{Bl}}
\newcommand{\Sing}{\operatorname{Sing}}
\newcommand{\Exc}{\operatorname{Exc}}

\newcommand{\bir}{\sim_{\mathrm{bir}}}
\newcommand{\ord}{\operatorname{ord}}
\newcommand{\I}{\mathcal I}
\newcommand{\pr}{\operatorname{pr}}
\newcommand{\Proj}{\operatorname{Proj}}

\newcommand{\Yt}{\widetilde{Y}}
\newcommand{\wtYm}{\widetilde Y_{m,d}}
\newcommand{\Ytil}{\widetilde{Y}}
\newcommand{\prim}{\mathrm{prim}}

\newcommand{\codim}{\operatorname{codim}}

\title[Unified counterexamples to Batyrev's conjecture]{A unified family of counterexamples to Batyrev's non-negativity conjecture on stringy Hodge numbers}
\author{Dimitrios I. Dais}
\address{Department of Mathematics, University of Thessaly,  Lamia,  35100 Greece}
\email{ddais@uth.gr}

\begin{document}

\begin{abstract} 
Recently, Huang and Satriano determined the sharp dimension threshold for Batyrev's non-negativity conjecture on stringy Hodge numbers: the conjecture holds in dimensions at most four and fails in every dimension at least five.  We show that their counterexamples fit into a single broad three-parameter family $Y_{m,d}\times(\PP^1)^n$ with
\[
Y_{m,d}:=V\!\left(\sum_{i=0}^{m}s_i g_i(x)+h(x)\right)\subset \PP^{m+d+1},
\qquad \deg g_i=d-1,\quad \deg h=d,
\]
whose singular locus is a linear space \(\Lambda\cong\PP^m\). Blowing up \(\Lambda\) gives a log resolution with one smooth exceptional divisor of discrepancy one. We compute the Hodge--Deligne polynomials of both the resolution and the exceptional divisor from their projective-space fibrations over $\PP^d$, derive a master formula for \(\Est(Y_{m,d}\times(\PP^1)^n)\), and prove a structure formula
\[
\hst^{p,q}(Y_{m,d}\times(\PP^1)^n)=A_{p,q}+B_{p,q}-C_{p,q},
\]
where the three terms are non-negative, so that the only possible source of negativity is $-C_{p,q}$, which is governed by the primitive middle cohomology of \(Z_{m,d}=V(g_0,\ldots,g_m)\subset\PP^d\). Off the diagonal every negative entry equals $-C_{p,q}$, a binomially weighted sum of shifted primitive Hodge numbers of \(Z_{m,d}\). On the diagonal the behaviour  is governed by the parity of \(m+d\). We also obtain a complete classification for \(0\le m\le d-1\), \(d\ge3\), and \(n\ge1\). The only non-counterexamples are \((m,d)=(0,3)\) for every \(n \ge 1\), and \((m,d)=(2,3)\) for \(n\ge4\).
\end{abstract}
\subjclass[2020]{Primary 14C30; Secondary 14B05, 14E15, 14J17, 14J45}
\maketitle
\vspace{-1cm}
\section{Introduction}
\noindent Let $Y$ be a normal $\QQ$-Gorenstein complex variety with log-terminal
singularities and let $f\colon X\to Y$ be a log resolution with exceptional locus
$\Exc(f)=\bigcup_{i\in I}D_i$ which consists of smooth prime divisors with only normal crossings and
$K_X=f^*K_Y+\sum_{i\in I}a_iD_i$ (where $a_i=a(D_i,Y)>-1$). Batyrev's \textit{stringy} $E$-\textit{function}
\cite[Def.~3.1]{Bat98} is
\begin{equation}
        E_{\text{st}}(Y;u,v)\;=\; \sum_{J\subseteq I}E(D_{J}^{\circ };u,v)\prod_{j\in
                J}\frac{uv-1}{(uv)^{a_{j}+1}-1}
        \label{eq:batdef}
\end{equation}%
(being independent of the particular choice of $f$) where $D_{J}^{\circ }=D_{J}\mathbb{r}\bigcup \nolimits_{j\in I\mathbb{r}J}D_{j}$ with $D_{J}=\bigcap \nolimits_{j\in J}D_{j}$ for $\varnothing \neq J\subseteq I%
\text{ and  }D_{\varnothing }=X$, $E(-;u,v)$ is the Hodge--Deligne polynomial,  $e(-)=E(-;1,1)$ is the \textit{ordinary Euler number},
\begin{equation}
e_{\text{st}}(Y):=\text{ }\underset{u,v\longrightarrow 1}{\lim }E_{\text{st}%
}(Y;u,v)=\sum_{J\subseteq I}e(D_{J}^{\circ })\prod \limits_{j\in J}\frac{1}{%
        a_{j}+1} \label{ESTRINGY}
\end{equation}%
the \textit{stringy Euler number}, and  $\text{ind}_{\text{st}}(Y)=\min \left \{ \left. \ell \in \mathbb{Z}_{\ge 1}%
\right \vert e_{\text{st}}(Y)\in 1/\ell \, \mathbb{Z}\right \} $  the so-called \textit{stringy index} of $Y$. If $Y$ has at worst Gorenstein canonical singularities, then the expected boundedness of $\text{ind}_{\text{st}}(Y)$ by a constant depending on the dimension (see {\cite[Conj.~5.9]{Bat98}}) is in general false. (For simple counterexamples see {\cite[Remark~1.9]{Da01}}.)  If $Y$ is projective with Gorenstein canonical singularities and
$\Est(Y;u,v)$ happens to be a polynomial, the \emph{stringy Hodge numbers}
$\hst^{p,q}(Y)$ are defined by
$\Est(Y;u,v)=\sum_{p,q}(-1)^{p+q}\hst^{p,q}(Y)u^pv^q$. A further conjecture of Batyrev is formulated as follows:

\begin{conjecture}[{\cite[Conj.~3.10]{Bat98}}]\label{conj:bat}
        Let $Y$ be a projective variety with at worst Gorenstein canonical
        singularities and assume $\Est(Y;u,v)$ is a polynomial.  Then
        $\hst^{p,q}(Y)\ge 0$ for all $p,q$.
\end{conjecture}

\noindent Important classes of such projective varieties for which Conjecture \ref{conj:bat} is known to hold include 
\medskip \newline 
\noindent $\bullet$  those admitting  crepant resolutions (\cite[Thm.~3.12]{Bat98}),
\smallskip \newline  
\noindent $\bullet$ those having at worst quotient singularities  (cf. \cite[Cor.~6.11]{BD96}, \cite[Thm.~7.5]{Bat99a}, \cite[\S 3]{CR04}, \cite[Remarks 1.4 (2) and 2.17, as well as Thm. 3.15 and Cor. 3.16]{Yas04}), 
\smallskip \newline 
\noindent $\bullet$  those having at worst toroidal singularities (\cite[Cor.~6.13]{BD96}),
\smallskip \newline  
\noindent $\bullet$ those which are projectivisations of determinantal varieties  {\small (\cite[Thm.~1.3 and Cor.~5.4]{CZ25}),}  and
\newline 
\noindent $\bullet$  those (a) of dimension $3$ or (b) of dimension $\ge4$ with at most isolated singularities under certain discrepancy hypotheses with respect to a log resolution {\small (\cite[Thm. 3.1]{SV07}, \cite[\S 3]{S09})}.\medskip

\noindent Moreover, Olano \cite[Thm. B]{Ola21} proved that $\hst^{p,1}(Y)\ge 0$ for all $p$ without any additional restriction. On the other hand,  Satriano and Usatine \cite{SU26} recently \textit{disproved} Conjecture~\ref{conj:bat} in
dimension~$7$; Huang and Satriano \cite{HS26} then settled the question
completely, proving the conjecture in dimension $\le4$ and giving, in
\cite[Thm.~B]{HS26}, counterexamples in every dimension $\ge 5$.

\begin{construction}We fix integers \(m\ge0\), \(d\ge3\), set $N=m+d+1,$ 
and choose general forms
$
g_0,\ldots,g_m\in \CC[x_0,\ldots,x_d]_{d-1},
\ h\in\CC[x_0,\ldots,x_d]_d.
$
In coordinates
\([s_0:\cdots:s_m:x_0:\cdots:x_d]\) on \(\PP^N\), we put (for $x=(x_0,...,x_d)$)
\begin{equation}\label{eq:Ymd}
Y_{m,d}:=V\!\left(\sum_{i=0}^m s_i g_i(x)+h(x)\right)\subset\PP^N,
\qquad
\Lambda:=V(x_0,\ldots,x_d)\cong\PP^m.
\end{equation}
Here $V(....)$ denotes the common zero locus of the listed forms. We shall also use
\begin{equation}\label{eq:CZ}
Z_{m,d}:=V(g_0,\ldots,g_m)\subset\PP^d,
\qquad
Z'_{m,d}:=V(g_0,\ldots,g_m,h)\subset\PP^d.
\end{equation}
For general equations these are smooth complete intersections of dimensions
\begin{equation}\label{eq:DIASTA}
k=d-m-1,
\qquad
l=d-m-2=k-1,
\end{equation}
whenever the expected dimensions are non-negative, and empty otherwise. (We use here Bertini's Theorem in the form of  \cite[Ch.~III, Corollary~10.9]{Har77}.) The main theorem is as follows: \end{construction}

\begin{theorem}\label{thm:classification}
        Let \(d\ge3\), \(0\le m\le d-1\), and \(n\ge1\). Then
        \(Y_{m,d}\times(\PP^1)^n\) is an $(m+d+n)$-dimensional projective variety with Gorenstein terminal singularities whose stringy $E$-function is a polynomial, and provides a counterexample to Batyrev's non-negativity conjecture \emph{\ref{conj:bat}} except precisely in the following cases:
        \begin{enumerate}[label=(\roman*)]
                \item \((m,d)=(0,3)\), for every \(n\ge1\);\smallskip
                \item \((m,d)=(2,3)\) and \(n\ge4\).
        \end{enumerate}
        For \((m,d)=(2,3)\) and \(n=1,2,3\), negative diagonal entries occur. If \(d=m+1\ge4\), a negative diagonal entry occurs for every \(n\ge1\).
\end{theorem}
\noindent \textbf{Notation and conventions.} We work over $\CC$.  $E(V;u,v)$ denotes the Hodge--Deligne polynomial,
a ring homomorphism $K_0(\mathrm{Var}_\CC)\to\ZZ[u,v]$. (Convention: $E(\varnothing;u,v):=0$.) By $\PP^{\nu}$ we denote the complex projective space of dimension $\nu\in \mathbb{Z}_{\ge 0}$. (Convention: $\PP^{-1}:=\varnothing$.)  Since $h^{p,q}(\PP^{\nu})=\delta_{p,q}$ for $0\le p,q \le \nu$, we shall throughout keep in mind the simple formula
\begin{equation}\label{eq:EPROJEC}
E(\PP^{\nu};u,v)=1+uv+\dots+(uv)^{\nu}. 
\end{equation}

\section{The one-divisor formula}\label{sec:one-divisor}

\begin{proposition}[One-divisor formula]\label{lem:one-divisor}
        Let \(Y\) be a normal projective \(\QQ\)-Gorenstein variety with log-terminal singularities. Assume that a log resolution \(f\colon X\to Y\) has exceptional locus equal to one smooth irreducible divisor \(D\), and
        \[
        K_X=f^*K_Y+D.
        \]
        Then
        \begin{equation}\label{eq:one-divisor}
                \Est(Y;u,v)=E(X;u,v)-\frac{uv}{1+uv}E(D;u,v).
        \end{equation}
\end{proposition}

\begin{proof}Here there is only one exceptional component, namely $D$.  Thus we may put $I=\{1\}$ in (\ref{eq:batdef}). There are only two subsets:
        $J=\varnothing$ and $J=\{1\}$. For $J=\varnothing$ we have
        $D_\varnothing^\circ=X\setminus D,$
        and there is no discrepancy factor.  For $J=\{1\}$ we have
        $D_{\{1\}}^\circ=D,$
        and because $a(D,Y)=1$ its factor is
        \[
        \frac{uv-1}{(uv)^{1+1}-1}
        =\frac{uv-1}{(uv)^2-1}
        =\frac{uv-1}{(uv-1)(uv+1)}
        =\frac{1}{1+uv}.
        \]
        Therefore
        \[
        \Est(Y;u,v)=E(X\setminus D;u,v)+\frac{1}{1+uv}E(D;u,v).\]
        The Hodge--Deligne polynomial is additive for a closed subvariety and its complement. Hence
        \[
        E(X;u,v)=E(X\setminus D;u,v)+E(D;u,v)
        \Longrightarrow
        E(X\setminus D;u,v)=E(X;u,v)-E(D;u,v).
        \]
        Substitution gives \eqref{eq:one-divisor}.
\end{proof}

\begin{corollary}[Clearing the denominator by multiplying by \((\PP^1)^n\)]\label{cor:prod}
        In the situation of Proposition~\emph{\ref{lem:one-divisor},} suppose in addition that $Y$ is
        Gorenstein.  Then $Y\times(\PP^1)^n$ is projective Gorenstein. The corresponding exceptional divisor $ D\times(\PP^1)^n$
   has the same discrepancy coefficient as $D$, namely $1$. In particular, the product has terminal singularities whenever $Y$ does, and for every $n\ge1$
        \begin{equation}\label{eq:prodformula}
                \Est(Y\times(\PP^1)^n;u,v)=(1+uv)^nE(X;u,v)-uv(1+uv)^{n-1}E(D;u,v)\;\in\;\ZZ[u,v].
        \end{equation}
\end{corollary}

\begin{proof}
        Since $(\PP^1)^n$ is smooth and projective, the product $       X\times(\PP^1)^n$
        is smooth and the morphism
        \[
        \theta:=f\times\mathrm{id}:X\times(\PP^1)^n\longrightarrow
        Y\times(\PP^1)^n
        \]
        is proper and birational.  Since $f$ is an isomorphism away from $D$, the exceptional locus of $\theta$ is $     D\times(\PP^1)^n,$ 
        which is again a smooth irreducible divisor.  Thus $\theta$ is a log resolution.
        
        Let $T=(\PP^1)^n$.  The canonical bundle of a product satisfies
        $K_{X\times T}=\operatorname{pr}_X^*K_X+\operatorname{pr}_T^*K_T,$
        and similarly for $Y\times T$.  Since
        $ K_X=f^*K_Y+D,$ we obtain
        \[
        K_{X\times T} =\operatorname{pr}_X^*(f^*K_Y+D)+\operatorname{pr}_T^*K_T =\theta^*K_{Y\times T}+D\times T.
        \]
        The discrepancy of $D\times T$ is again $1$.  Hence the product resolution has the same discrepancy coefficient on the corresponding exceptional divisor; in particular, terminality is preserved under product with the smooth variety $T$. Projectivity is preserved under products. The Gorenstein property is also preserved because $Y$ is Gorenstein and $T$ is smooth; indeed, the dualising sheaf of a product is the exterior tensor product of the dualising sheaves. Now we apply Proposition \ref{lem:one-divisor} to $Y\times(\PP^1)^n$.  Multiplicativity of $E$ gives
        \[E(X\times T;u,v)=E(X;u,v)E(T;u,v),
        \ 
        E(D\times T;u,v)=E(D;u,v)E(T;u,v).\]  Since $   E(\PP^1;u,v)=1+uv,
        \  E(T;u,v)=(1+uv)^n,$
        we get
                \begin{equation*}
                                E_{\text{st}}(Y\times(\PP^1)^n;u,v)=E(X;u,v)(1+uv)^{n}-\tfrac{uv}{1+uv}%
                                E(D;u,v)(1+uv)^{n}
                \end{equation*}
                which can be rewritten as \eqref{eq:prodformula}.
\end{proof}
\section{The unified family: geometry, resolution, and Hodge--Deligne polynomials}\label{sec:geometry}
\begin{proposition}\label{Prop1}
        Let \(m\ge0\), \(d\ge3\), and let \(Y_{m,d}\) be as in \eqref{eq:Ymd}. Then
        $Y_{m,d}$ is an irreducible normal Gorenstein hypersurface of dimension \(m+d\), and
                \[
                \omega_{Y_{m,d}}\cong\mathcal O_{Y_{m,d}}(-m-2),
                \qquad
                \Sing(Y_{m,d})=\Lambda,
                \qquad
                \ord_\Lambda(Y_{m,d})=d-1.
                \]Moreover,  $Y_{m,d}$ is  a (singular) Fano variety. 
\end{proposition}
\begin{proof} We prove the assertions separately. \medskip \newline $\bullet $ $Y_{m,d}$ \textit{is an irreducible hypersurface.} The defining polynomial $F(s,x)=\sum_{i=0}^{m}s_{i}g_{i}(x)+h(x)$ is homogeneous of total degree $d$. $Y_{m,d}$ is a degree-$d$ hypersurface of dimension $N-1=m+d$. We set $R:=\mathbb{C}[x_{0},\dots ,x_{d}]$ and consider $F=F(s,x)$ as an element of $R[s_{0},...,s_{m}]$ (with $x_{0},...,x_{d}$ regarded as coefficients). The \textit{total} $s$-degree (where $s=(s_{0},...,s_{m})$) is additive under multiplication and $\deg _{s}(F(s,x))$  over $R$ equals $1$ because for $i\in \{0,...,m\}$\begin{equation*} \left[ \deg _{s}(s_{i}g_{i}(x))=\deg _{s}(s_{i})+\deg _{s}(g_{i}(x))=1+0=1,\ \deg _{s}(h(x))=0\right] \Longrightarrow \deg _{s}(F(s,x))=1. \end{equation*}Suppose that $F(s,x)$ is reducible. Then there exists a non-trivial factorisation, say\begin{equation} F(s,x)=\Phi (s,x)\Psi (s,x)\text{ \  \ in \ }R[s_{0},...,s_{m}]=\mathbb{C}[s_{0},...,s_{m},x_{0},\dots ,x_{d}].  \label{FACTORISATION} \end{equation}Since $R$ is an integral domain,  \[
        \begin{array}{l}
                \deg _{s}(\Phi (s,x)\Psi (s,x))=\deg _{s}(\Phi (s,x))+\deg _{s}(\Psi
                (s,x))=\deg _{s}(F(s,x))=1\medskip  \\ 
                \Longrightarrow (\deg _{s}(\Phi (s,x)),\deg _{s}(\Psi (s,x)))\in
                \{(0,1),(1,0)\}.%
        \end{array}%
        \]
       After interchanging $\Phi$ and $\Psi,$ if necessary, we may assume that $\deg _{s}(\Phi (s,x))=0,$ i.e., that $\Phi (s,x)=\Phi (x)\in $ $R$ (involving no $s_{i}$'s) is a non-constant polynomial. Since $\deg _{s}(\Psi (s,x))=1,$ the most general form of $\Psi (s,x)$ is\begin{equation*} \Psi (s,x)=\psi _{0}(x)s_{0}+\cdots +\psi _{m}(x)s_{m}+\psi (x),\text{ \ where \ }\psi _{0}(x),...,\psi _{m}(x),\psi (x)\in R. \end{equation*}Multiplying out $\Phi (s,x)\Psi (s,x)$ gives \begin{equation*} F(s,x)=\sum_{i=0}^{m}s_{i}g_{i}(x)+h(x)=\sum_{i=0}^{m}s_{i}\Phi (x)\psi _{i}(x)+\Phi (x)\psi (x), \end{equation*}and since $1,s_{0},...,s_{m}$ are linearly independent over $R,$ by comparison of coefficients we obtain\begin{equation*} \left.  \begin{array}{l} g_{i}(x)=\Phi (x)\psi _{i}(x),\  \forall i\in \{0,...,m\} \medskip \\  h(x)=\Phi (x)\psi (x)\end{array}\right \} \Longrightarrow \Phi (x)\mid \text{gcd}(g_{0}(x),...,g_{m}(x),h(x)). \end{equation*}By assumption, $g_{0}(x),...,g_{m}(x),h(x)$ are sufficiently general. Let us fix a $d'\in\{1,...,d-1\}$. The tuples $(g_0(x),....,g_m(x),h(x))$ admitting a common homogeneous factor of degree $d'$ form a proper closed subset of the parameter space (it is the image of a projective bundle of smaller dimension). Since $d'$ ranges over a finite set, the union of these loci is a proper closed subset, and a general tuple lies \textit{outside} it.  So for general choices, gcd$(g_{0}(x),...,g_{m}(x),h(x))=1$ and  $\Phi \left( x\right) $ would be necessarily constant, contradicting the assumption that the factorisation \eqref{FACTORISATION} is non-trivial.  Thus $Y_{m,d}$ is irreducible.
        
        \medskip \noindent $\bullet $ \textit{Canonical sheaf and Gorenstein property.} A hypersurface is a local complete intersection, hence Gorenstein. Adjunction (\cite[Ch.~III, Thm.~7.11]{Har77}) for degree-$d $ hypersurfaces in $\mathbb{P}^{N}$ gives  \begin{equation*}
                \left. 
                \begin{array}{l}
                        \omega _{Y_{m,d}}\cong \bigl(\omega _{\mathbb{P}^{N}}\otimes \mathcal{O}_{%
                                \mathbb{P}^{N}}(d)\bigr)|_{Y_{m,d}}\medskip  \\ 
                        \omega _{\mathbb{P}^{N}}\cong \mathcal{O}_{\mathbb{P}^{N}}(-N-1)%
                \end{array}%
                \right \} \Longrightarrow \omega _{Y_{m,d}}\cong \mathcal{O}_{Y_{m,d}}(d-N-1)=%
                \mathcal{O}_{Y_{m,d}}(-m-2).
        \end{equation*}
        
        \noindent $\bullet $ \textit{Every point of $\Lambda$ is singular}. The partial derivatives are  \begin{equation*} \frac{\partial F}{\partial s_{i}}=g_{i}(x),\text{ }\forall i\in \{0,...,m\},\  \  \text{and \ }\frac{\partial F}{\partial x_{j}}=\sum_{i=0}^{m}s_{i}\frac{\partial g_{i}}{\partial x_{j}}(x)+\frac{\partial h}{\partial x_{j}}(x),\  \forall j\in \{0,...,d\}. \end{equation*}Along $\Lambda $, all $x_{j}$ vanish; since $\deg \left( g_{i}\right) \geq 2$ and $\deg \left( h\right) \geq 3$,  \begin{equation*} g_{i}(0)=0,\  \  \frac{\partial g_{i}}{\partial x_{j}}(0)=0,\text{ \ and\  }\frac{\partial h}{\partial x_{j}}(0)=0,\  \  \forall (i,j)\in \{0,...,m\} \times \{0,...,d\}. \end{equation*}Hence all first derivatives of $F(s,x)$ vanish along $\Lambda $ and therefore   $\Lambda \subseteq $ Sing$(Y_{m,d})$.
        
        \medskip \noindent $\bullet $ \textit{There are no singular points away from $\Lambda$.} Conversely, suppose $[s:x]\in \Sing(Y_{m,d})$ with $x\neq 0$. Then all $g_{i}(x)=0$, and because $[s:x]\in Y_{m,d}$, also $h(x)=0$; thus $[x]\in Z_{m,d}^{\prime }$. The remaining singularity equations give  \begin{equation} \sum_{i=0}^{m}s_{i}\nabla g_{i}(x)+\nabla h(x)=0.  \label{NABLES} \end{equation}By the Jacobian criterion (\cite[Ch.~I, \S 5, ex. 5.8, p. 37]{Har77}) for the smooth complete intersection $Z_{m,d}^{\prime }$ at $[x]$ the Jacobian matrix  \begin{equation*} J_{Z_{m,d}^{\prime }}(x)={\footnotesize \left(  \begin{array}{ccc} \frac{\partial g_{0}}{\partial x_{0}}\left( x\right)  & \cdots  & \frac{\partial g_{0}}{\partial x_{d}}\medskip \left( x\right)  \\  \vdots  &  & \vdots  \\  \frac{\partial g_{m}}{\partial x_{0}}\left( x\right)  \medskip  & \cdots  & \frac{\partial g_{m}}{\partial x_{d}}\left( x\right)  \medskip \\  \frac{\partial h}{\partial x_{0}}\left( x\right)  & \cdots  & \frac{\partial h}{\partial x_{d}}\left( x\right)  \end{array}\right) } \end{equation*}must have rank $m+2.$ Equivalently, its defining gradients are linearly independent. \eqref{NABLES} is impossible because the coefficient of $\nabla h(x)$ is one. If $Z_{m,d}^{\prime }=\varnothing $, no such $[x]$ exists in the first place, and the assertion is vacuous. Thus $\Sing(Y_{m,d})=\Lambda $.\medskip
        
        \noindent $\bullet $ \textit{Multiplicity along }$\Lambda $\textit{.} Let $I_{\Lambda }=(x_{0},\ldots ,x_{d})$. For positive integers $r$ the power $I_{\Lambda }^{r}$ consists of polynomials whose monomials have total $x$-degree at least $r.$ (The variables $s_{i}$ are not counted.) The $I_{\Lambda }$-adic order of $F$ is defined as ord$_{\Lambda }(F)=\max \left \{ r:F\in I_{\Lambda }^{r}\right \} .$ (See \cite[Ch. 5]{Eis95} and \cite[Ch.~II, \S 7]{Har77}.) Each  $g_{i}(x)$ is homogeneous of degree $d-1$ in $x_{0},...,x_{d}$ and every monomial of $g_{i}(x)$ contains exactly $d-1$ $x$-factors. Therefore, $g_{i}(x)\in I_{\Lambda }^{d-1}.$ Multiplying by $s_{i}$ does not change the number of $x$'s, so\begin{equation*} \left.  \begin{array}{r} \left[ s_{i}g_{i}(x)\in I_{\Lambda }^{d-1},\text{ }\forall i\in \{0,...,m\}\right] \Longrightarrow \sum_{i=0}^{m}s_{i}g_{i}(x)\in I_{\Lambda }^{d-1}\medskip  \\  h(x)\in I_{\Lambda }^{d}\subset I_{\Lambda }^{d-1}\end{array}\right \} \Longrightarrow F(s,x)\in I_{\Lambda }^{d-1}. \end{equation*}Splitting $F(s,x)$ according to its $x$-degree, we see that  
        \[
        F(s,x)=\underset{x\text{-degree }d-1}{\underbrace{\sum%
                        \nolimits_{i=0}^{m}s_{i}g_{i}(x)}}+\underset{x\text{-degree }d}{\underbrace{%
                        h(x)}}\in I_{\Lambda }^{d-1}\mathbb{r}I_{\Lambda }^{d}\Longrightarrow
        \left \{ 
        \begin{array}{c}
                \text{in}_{I_{\Lambda }}(F(s,x))=\sum_{i=0}^{m}s_{i}g_{i}(x)\medskip \text{ }
                \\ 
                \text{(non-zero piece in the }I_{\Lambda }\text{-adic filtration)}%
        \end{array}%
        \right \} 
        \]%
        and consequently that ord$_{\Lambda }(F)=d-1.$ Geometrically, this says that the hypersurface $Y_{m,d}$ has \textit{multiplicity} $d-1$ \textit{along the whole subspace} $\Lambda.$  \medskip \newline
        \noindent $\bullet $ \textit{Normality.} The singular locus is $\Lambda \cong \mathbb{P}^m$ while $\text{dim}(Y_{m,d})=m+d$. Therefore \begin{equation*}
                \text{codim}_{Y_{m,d}}(\text{Sing}(Y_{m,d}))=d\geq 3.
        \end{equation*}%
        In particular, $Y_{m,d}$ is regular in codimension $1.$ Moreover,  since the ideal sheaf of $Y_{m,d}$ in $\mathbb{P}^{N}$ is generated by one element at every point, $Y_{m,d}$ is Cohen-Macaulay. This implies the normality of $Y_{m,d}$. (See  \cite[Thm. 11.5]{Eis95},  \cite[Ch.~II, Proposition~8.23]{Har77} or \cite[Thm 23.8, p. 183]{Mat89}.) \medskip \newline
         \noindent $\bullet $ \textit{$Y_{m,d}$  is Fano.} $\mathcal O_{Y_{m,d}}(1)$ is the restriction of the hyperplane bundle $\mathcal O_{\mathbb{P}^N}(1)$, hence it is very ample. Since $m+2\ge 2$, the positive tensor power $\mathcal O_{Y_{m,d}}(m+2)$ is again very ample, and therefore ample. Thus $-K_{Y_{m,d}}$ is ample.
        \end{proof}

\begin{theorem}\label{thm:uniform-geometry}
Let \(m\ge0\), \(d\ge3\), and let \(Y_{m,d}\) and $\Lambda$ be as in \eqref{eq:Ymd}. Then:
\begin{enumerate}[label=(\roman*)]
\item The blow-up
\[
f\colon \widetilde Y_{m,d}:=\Bl_\Lambda Y_{m,d}\longrightarrow Y_{m,d}
\]
is a log resolution. Its exceptional locus is a single smooth irreducible divisor \(D_{m,d}\), and
\[
K_{\widetilde Y_{m,d}}=f^*K_{Y_{m,d}}+D_{m,d}.
\]
Hence \(a(D_{m,d},Y_{m,d})=1\) and \(Y_{m,d}\) is terminal.
\item There is a morphism
\(\rho\colon\widetilde Y_{m,d}\to\PP^d\) whose fibre is \(\PP^m\) over \(\PP^d\setminus Z'_{m,d}\) and \(\PP^{m+1}\) over \(Z'_{m,d}\). The induced map \(D_{m,d}\to\PP^d\) has fibre \(\PP^{m-1}\) over \(\PP^d\setminus Z_{m,d}\) and \(\PP^m\) over \(Z_{m,d}\). Consequently,
\begin{align}
E(D_{m,d};u,v)&=E(\PP^{m-1};u,v)E(\PP^d;u,v)+(uv)^mE(Z_{m,d};u,v),\label{eq:ED}\\
E(\widetilde Y_{m,d};u,v)&=E(\PP^m;u,v)E(\PP^d;u,v)+(uv)^{m+1}E(Z'_{m,d};u,v).\label{eq:EYtilde}
\end{align}
\item For every \(n\ge1\),
\begin{align}
\Est\bigl(Y_{m,d}\times(\PP^1)^n;u,v\bigr)
={}&(1+uv)^n\bigl(E(\PP^m;u,v)E(\PP^d;u,v)+(uv)^{m+1}E(Z'_{m,d};u,v)\bigr)\label{eq:uniform-Est-raw}\\
&-uv(1+uv)^{n-1}\bigl(E(\PP^{m-1};u,v)E(\PP^d;u,v)+(uv)^mE(Z_{m,d};u,v)\bigr).\nonumber
\end{align}
In particular, this is a polynomial in \(u,v\).
\end{enumerate}
\end{theorem}

\begin{proof}
(i) Outside $\Lambda$, at least one $x_i$ is non-zero, so one may form the projective point
\[
[x_0:\cdots:x_d]\in \PP^d.
\]
This defines the rational map
$
\phi:\PP^N\dashrightarrow \PP^d,
\ \
[s:x]\longmapsto [x_0:\cdots:x_d].
$
It is undefined precisely on $\Lambda$, because at a point of $\Lambda$ one would obtain
$
[0:\cdots:0],
$
which is not a point of projective space.  Hence
$
\operatorname{Indet}(\phi)=\Lambda.
$
Equivalently, $\phi$ is the rational map associated with the $d+1$ global sections
\[
x_0,\ldots,x_d\in H^0(\PP^N,\OO_{\PP^N}(1)),
\]
whose common zero locus is exactly $\Lambda$. The ideal sheaf of the centre is
$
\I_{\Lambda}
$ determined by $(x_0,\ldots,x_d)$.
 Let
$
\pi:\Bl_{\Lambda}\mathbb{P}^N\to\mathbb{P}^N
$
be the blow-up of the ambient projective space, and let $\Xi$ be its exceptional divisor.  The blow-up is the relative Proj of the Rees algebra:
\[
\text{Bl}_{\Lambda}\PP^N
=
\Proj_{\PP^N}
\left(
\bigoplus_{n\ge 0}\I_{\Lambda}^n
\right).
\]
One of the basic features of the blow-up is that the pullback of the base ideal becomes invertible:
\[
\I_{\Lambda}\OO_{\text{Bl}_{\Lambda}\PP^N}\cong \OO_{\text{Bl}_{\Lambda}\PP^N}(-\Xi).
\]
Thus the common zero locus of $x_0,\ldots,x_d$ is replaced by the exceptional divisor; after removing $\Xi$ from the pulled-back system one is left with a base-point-free linear system on $\text{Bl}_{\Lambda}\PP^N$.  Consequently the rational map $\phi$ becomes an honest morphism
$
\rho: \text{Bl}_{\Lambda}\PP^N\longrightarrow \PP^d.
$
If we write $\mathbb{P}^{N}=\mathbb{P}(W\oplus W^{\prime })$ with 
\begin{equation*}
	\text{dim}(W)=m+1,\text{ \ dim}(W^{\prime })=d+1,\  \  \Lambda =\mathbb{P}%
	(W),\  \  \mathbb{P}(W^{\prime })=\mathbb{P}^{d},
\end{equation*}%
projection from $\Lambda $ to $\mathbb{P}(W^{\prime })$ induces the above morphism $\rho .$ More precisely, 
\begin{equation*}
	\mathbb{P}_{\mathbb P^{d}}(\mathcal{O}_{\mathbb{P}^{d}}{}^{\oplus (m+1)}\oplus \mathcal{O}_{%
		\mathbb{P}^{d}}(-1))\cong \text{Bl}_{\Lambda }\mathbb{P}^{N}\overset{\rho }{%
		\longrightarrow }\mathbb{P}^{d}
\end{equation*}
so that $\rho$ is a $\mathbb{P}^{m+1}$-bundle. On the other hand, 
\begin{equation*}
	\mathcal{N}_{\Lambda /\mathbb{P}^{N}}\cong \left( \left. \mathcal{O}_{%
		\mathbb{P}^{N}}(1)\right \vert _{\Lambda }\right) ^{\oplus (d+1)}\cong 
	\mathcal{O}_{\mathbb{P}^{m}}(1)^{\oplus (d+1)},
\end{equation*}%
and
$\Xi \cong \mathbb{P}(%
	\mathcal{N}_{\Lambda /\mathbb{P}^{N}})\cong \Lambda \times \mathbb{P}%
	^{d}\cong \mathbb{P}^{m}\times \mathbb{P}^{d}.$
Concretely, for a fixed $[x]\in\mathbb{P}(W^{\prime })=\mathbb{P}^{d}$, after choosing a representative $x=(x_0,...,x_d)$, the fibre $\rho ^{-1}([x])\cong \mathbb P^{m+1}$ has homogeneous coordinates $[s_{0}:\cdots :s_{m}:u]$. Under the morphism $%
\pi $, such a point maps to $[s_{0}:\cdots :s_{m}:ux_{0}:\cdots :ux_{d}]\in \mathbb{P}^{N}.$ Moreover,
$\Xi =\{u=0\}.$ On $\text{Bl}_{\Lambda}\PP^N\setminus \Xi$, where $\pi$ is an isomorphism onto $\PP^N\setminus\Lambda$, one has
$
\rho=\phi\circ\pi.
$
Thus $\text{Bl}_{\Lambda}\PP^N$ is the natural resolution of the indeterminacy of projection from $\Lambda$.
  We define $H:=\pi ^{\ast }\OO_{\PP^{N}}(1).$ The hyperplanes in $\PP%
^{N}$ containing $\Lambda $ are generated by $x_{0},\ldots ,x_{d}$; their
strict transforms define $\rho $. Hence 
\begin{equation}\label{HXI}
        H-\Xi =\rho ^{\ast }\OO_{\PP%
                ^{d}}(1).
\end{equation}
Equivalently, the map $(\pi ,\rho ):$ Bl$_{\Lambda }\PP%
^{N}\hookrightarrow \PP^{N}\times \PP^{d}$ realises $H$ as $\OO(1,0)|_{\text{%
                Bl}_{\Lambda }\PP^{N}}$ and $H-\Xi $ as $\OO(0,1)|_{\text{Bl}_{\Lambda }\PP%
        ^{N}}$. Since $Y_{m,d}$ has degree $d$ and vanishes along $\Lambda $ to order $d-1$,
its strict transform satisfies 
\begin{equation}
        \wtYm \sim dH-(d-1)\Xi =H+(d-1)(H-\Xi ).  \label{eq:class-strict}
\end{equation}%
By \eqref{HXI}, 
\[
\OO_{\text{Bl}_{\Lambda }\PP^{N}}\bigl(dH-(d-1)\Xi \bigr)\cong (\pi ,\rho
)^{\ast }\OO_{\PP^{N}\times \PP^{d}}(1,d-1).
\]%
Since $d-1\geq 2$, $\OO(1,d-1)$ is very ample on the product, and its
restriction to Bl$_{\Lambda }\PP^{N}$ is very ample. It is also useful to
see directly that the forms used in the construction fill the whole complete
linear system. The standard push-forward identity for a blow-up gives 
\[
H^{0}\bigl(\text{Bl}_{\Lambda }\PP^{N},\OO_{\text{Bl}_{\Lambda }\PP%
        ^{N}}(dH-(d-1)\Xi )\bigr)\cong H^{0}\bigl(\PP^{N},\I_{\Lambda}^{d-1}(d)\bigr)%
.
\]%
A degree-$d$ monomial lying in $I_{\Lambda}^{d-1}$ (with $I_{\Lambda }=(x_{0},\ldots ,x_{d})$ as in Proposition \ref{Prop1})  has $x$-degree  at least $d-1$ (counting multiplicities). Since its total degree is $d$, it is of
exactly one of the following two types:

\begin{itemize}
        \item degree $d$ in the $x$'s, contributing to $h(x)$; 
        
        \item degree $d-1$ in the $x$'s and degree $1$ in the $s$'s, contributing to 
        $\sum_{i=0}^{m} s_i g_i(x)$.
\end{itemize}
\noindent Thus a general choice of $(g_{0},\ldots ,g_{m},h)$ gives a general member of
this very ample complete linear system. Since Bl$_{\Lambda }\PP^{N}$ is
smooth, Bertini implies that $\wtYm$ is smooth and irreducible.

Restricting the leading term of $F(s,x)$ to $\Xi $ gives 
\begin{equation}
        D_{m,d}=\wtYm \cap \Xi =\left \{ ([s],[y])\in \PP^{m}\times \PP%
        ^{d}\; \left \vert \sum_{i=0}^{m}s_{i}g_{i}(y)=0\right. \right \} .
        \label{eq:D-equation}
\end{equation}%
This is a divisor of bidegree $(1,d-1)$, i.e. a member of $|\OO_{\PP%
        ^{m}\times \PP^{d}}(1,d-1)|.$ The bundle $\OO(1,d-1)$ is very ample, and the
section in \eqref{eq:D-equation} is general when the $g_{i}$ are general.
Bertini therefore gives that $D_{m,d}$ is smooth and irreducible. Since $%
f:=\pi |_{\wtYm}$ is an isomorphism away from $\Lambda $ and $f^{-1}(\Lambda
)=D_{m,d},$ we have $\Exc(f)=D_{m,d}.$ A single smooth divisor is
automatically a simple normal crossings divisor. Therefore $f$ is a log
resolution. The centre $\Lambda \cong \PP^{m}$ has codimension 
\[
\codim_{\PP^{N}}\Lambda =N-m=d+1.
\]%
Hence the canonical divisor of the ambient blow-up is 
$
        K_{\text{Bl}_{\Lambda }\PP^{N}}=\pi ^{\ast }K_{\PP^{N}}+d\,\Xi . 
$
Using \eqref{eq:class-strict} and adjunction on the smooth divisor $\wtYm%
\subset $ Bl$_{\Lambda }\PP^{N}$, we obtain 
\[
K_{\wtYm}=(K_{\text{Bl}_{\Lambda }\PP^{N}}+\wtYm)|_{\wtYm}=\bigl(\pi ^{\ast
}K_{\PP^{N}}+d\,\Xi +\pi ^{\ast }Y_{m,d}-(d-1)\Xi \bigr)|_{\wtYm}=\bigl(\pi
^{\ast }(K_{\PP^{N}}+Y_{m,d})+\Xi \bigr)|_{\wtYm}.
\]%
Adjunction on $Y_{m,d}$ gives $K_{Y_{m,d}}=(K_{\PP^{N}}+Y_{m,d})|_{Y_{m,d}},$
while $\Xi |_{\wtYm}=D_{m,d}$. Thus 
\begin{equation}
        K_{\wtYm}=f^{\ast }K_{Y_{m,d}}+D_{m,d}.  \label{eq:discrepancy}
\end{equation}%
Therefore $a(D_{m,d},Y_{m,d})=1>0.$ Since $f$ is a log resolution and every
exceptional divisor on this resolution has positive discrepancy (there is
only one), the standard discrepancy criterion implies that $Y_{m,d}$ is
terminal. \medskip

\noindent (ii) First, let us give an explicit model for the fibres of $\rho :$ Bl$%
_{\Lambda }\PP^{N}\longrightarrow \PP^{d}.$ One describes Bl$_{\Lambda }\PP%
^{N}$ (up to isomorphism) as the closure of the graph of $\phi$, i.e. as the incidence variety%
\[
\text{Bl}_{\Lambda }\PP^{N}\cong \left \{ \left. \left( \left[ s:x\right]
,[y] \right) \in \PP^{N}\times \PP^{d}\right \vert x_{i}y_{j}=x_{j}y_{i}\text{
        for all }  i,j \in \{0,..,d\} \right \} .
\]%
Now we fix a  $[y]\in \PP^{d}$. Every point in the fibre $\rho ^{-1}([y])$
has $x$ proportional to $y.$ Hence there is a scalar $\tau $ such that $%
x_{i}=\tau y_{i}$ for all $i\in \{0,...,d\}.$ (We shall abbreviate it by
writing $x=\tau y.$) A point of the original projective space
is $\left[ s_{0}:\cdots :s_{m}:x_{0}:\cdots :x_{d}\right] .$ In $\rho
^{-1}([y]),$ since $x=\tau y,$ this becomes $\left[ s_{0}:\cdots :s_{m}:\tau
y_{0}:\cdots :\tau y_{d}\right] .$ With $y$ fixed, the only varying
quantities are $s_{0},...,s_{m},\tau .$ Therefore $\rho ^{-1}([y])\cong \PP%
^{m+1}$ with homogeneous coordinates $\left[ s_{0}:\cdots :s_{m}:\tau \right]
.$ In addition, 
\[
F(s,x)=\sum \limits_{i=0}^{m}s_{i}g_{i}(x)+h(x)\Longrightarrow F(s,\tau
y)=\sum \limits_{i=0}^{m}s_{i}g_{i}(\tau y)+h(\tau y)
\]%
with $g_{i}(\tau y)=\tau ^{d-1}g_{i}(y),$ $h(\tau y)=\tau ^{d}h(y).$ Hence 
\[
F(s,\tau y)=\sum_{i=0}^{m}s_{i}g_{i}(\tau y)+h(\tau y)=\tau ^{d-1}\left(
\sum_{i=0}^{m}s_{i}g_{i}(y)+\tau h(y)\right) .
\]%
The factor $\tau ^{d-1}$ is the exceptional multiplicity; after removing it,
the strict transform is 
\begin{equation}
        \wtYm=V\left( \sum_{i=0}^{m}s_{i}g_{i}(y)+\tau h(y)\right) .
        \label{eq:fibre-linear}
\end{equation}%
This is one linear equation in $\PP^{m+1}$. If $[y]\notin Z_{m,d}^{\prime }$%
, the coefficient vector $(g_{0}(y),\ldots ,g_{m}(y),h(y))$ is non-zero, so %
\eqref{eq:fibre-linear} cuts out a hyperplane $\PP^{m}$. If $[y]\in
Z_{m,d}^{\prime }$, every coefficient is zero, so the equation vanishes
identically and the fibre is all of $\PP^{m+1}$. Thus 
\[
\rho ^{-1}(y)\cap \wtYm \cong 
\begin{cases}
        \PP^m,& \text{if}\ [y]\notin Z_{m,d}^{\prime },\medskip\\
        \PP^{m+1},& \text{if}\ [y]\in Z_{m,d}^{\prime }.
\end{cases}
\]%
For the exceptional divisor we impose $\tau =0$. 
\eqref{eq:fibre-linear} becomes $\sum_{i=0}^{m}s_{i}g_{i}(y)=0$ in $%
[s_{0}:\cdots :s_{m}]\in \PP^{m}$. If $y\notin Z_{m,d}$, this is a non-zero
linear equation and hence gives $\PP^{m-1}$; if $y\in Z_{m,d}$, it vanishes
identically and the fibre is $\PP^{m}$. (The convention $\PP%
^{-1}=\varnothing $ makes the statement correct also for $m=0$.) 

Over each of the two strata in question, the families just described are
projectivisations of vector bundles (equivalently, after choosing a non-zero
coefficient they are Zariski locally trivial projective-space fibrations).
Hence multiplicativity of $E$ applies.

For $\wtYm$, 
\begin{align*}
        E(\wtYm;u,v)& =E(\PP^{m};u,v)E(\PP^{d}\setminus Z_{m,d}^{\prime };u,v)+E(\PP%
        ^{m+1};u,v)E(Z_{m,d}^{\prime };u,v) \\
        & =E(\PP^{m};u,v)\bigl(E(\PP^{d};u,v)-E(Z_{m,d}^{\prime };u,v)\bigr)+E(\PP%
        ^{m+1};u,v)E(Z_{m,d}^{\prime };u,v) \\
        & =E(\PP^{m};u,v)E(\PP^{d};u,v)+\bigl(E(\PP^{m+1};u,v)-E(\PP^{m};u,v)\bigr)%
        E(Z_{m,d}^{\prime };u,v) \\
        & =E(\PP^{m};u,v)E(\PP^{d};u,v)+(uv)^{m+1}E(Z_{m,d}^{\prime };u,v).
\end{align*}%
Likewise, 
\begin{align*}
        E(D_{m,d};u,v)& =E(\PP^{m-1};u,v)E(\PP^{d}\setminus Z_{m,d};u,v)+E(\PP%
        ^{m};u,v)E(Z_{m,d};u,v) \\
        & =E(\PP^{m-1};u,v)\bigl(E(\PP^{d};u,v)-E(Z_{m,d};u,v)\bigr)+E(\PP%
        ^{m};u,v)E(Z_{m,d};u,v) \\
        & =E(\PP^{m-1};u,v)E(\PP^{d};u,v)+\bigl(E(\PP^{m};u,v)-E(\PP^{m-1};u,v)\bigr)%
        E(Z_{m,d};u,v) \\
        & =E(\PP^{m-1};u,v)E(\PP^{d};u,v)+(uv)^{m}E(Z_{m,d};u,v).
\end{align*}%
(iii) By \eqref{eq:discrepancy}, the resolution has one exceptional divisor
of discrepancy one. Therefore applying equality \eqref{eq:prodformula} of Corollary \ref{cor:prod} for $Y=Y_{m,d}$, $X=\wtYm$ and $D=D_{m,d}$  we obtain
\[
E_{\text{st}}(Y_{m,d}\times (\PP^{1})^{n};u,v)=(1+uv)^{n}E(\wtYm%
;u,v)-uv(1+uv)^{n-1}E(D_{m,d};u,v).
\]%
Substituting \eqref{eq:ED} and \eqref{eq:EYtilde} yields exactly %
\eqref{eq:uniform-Est-raw}. Since $n\geq 1$, every factor displayed is a
polynomial in $u,v$, so the result belongs to $\mathbb{Z}[u,v]$.
\end{proof}

\begin{proposition}\label{PROFANO}
	For every $n\ge 1$,
	\[
		\omega_{Y_{m,d}\times (\PP^{1})^{n}}
		\simeq
		\pr_0^*\cO_{Y_{m,d}}(-m-2)
		\otimes
		\bigotimes_{i=1}^{n}\pr_i^*\cO_{\PP^1}(-2),
	\]
	where $\pr_0:Y_{m,d}\times (\PP^{1})^{n}\to Y_{m,d}$ and $\pr_i:Y_{m,d}\times (\PP^{1})^{n}\to\PP^1$ are the projections.
	Consequently $-K_{Y_{m,d}\times (\PP^{1})^{n}}$ is ample (indeed very ample), and $Y_{m,d}\times (\PP^{1})^{n}$ is Fano.
\end{proposition}

\begin{proof}
	For products of Gorenstein varieties (and in particular for a Gorenstein
	variety times a smooth variety), the canonical line bundle is the exterior
	tensor product of the canonical line bundles of the factors. Hence
	\[
	\omega_{Y_{m,d}\times (\PP^{1})^{n}}
	\simeq
	\pr_0^*\omega_{Y_{m,d}}
	\otimes
	\bigotimes_{i=1}^{n}\pr_i^*\omega_{\PP^1}.
	\]
	We know that
	$
	\omega_{Y_{m,d}}\simeq\cO_{Y_{m,d}}(-m-2),
	\ \text{and} \
	\omega_{\PP^1}\simeq\cO_{\PP^1}(-2).
	$
The external product of two very ample	line bundles is very ample on the product. (Geometrically this follows by taking the product embeddings and then
the Segre embedding.) This generalises inductively for finitely many very ample line bundles. Therefore $-K_{Y_{m,d}\times (\PP^{1})^{n}}$ is very ample, in particular ample.
\end{proof}

\begin{proposition} The anticanonical degrees of $Y_{m,d}$ and $Y_{m,d}\times (\PP^{1})^{n}$ are
\begin{equation}
	(-K_{Y_{m,d}})^{m+d}\;=\;d\,(m+2)^{m+d}, \label{antican1}
\end{equation}
	and, respectively, 
\begin{equation}
	(-K_{Y_{m,d}\times (\PP^{1})^{n}})^{m+d+n}\;=\;2^{\,n}\,d\,(m+2)^{m+d}\,\frac{(m+d+n)!}{(m+d)!}. \label{antican2}
\end{equation}
\end{proposition}

\begin{proof}
Let $\mathfrak{H}:=c_{1}(\mathcal{O}_{Y_{m,d}}(1))$ be the hyperplane class
on $Y_{m,d},$ obtained by restricting the hyperplane class of $\mathbb{P}%
^{m+d+1}.$ Since $Y_{m,d}$ is a hypersurface of degree $d,$ $\deg
(Y_{m,d})=d=\mathfrak{H}^{m+d}.$ (Geometrically, intersecting $Y_{m,d}$ with 
$m+d$ sufficiently general hyperplanes leaves $d$ points, counted with
multiplicity.) Thus  \eqref{antican1} is true because
\begin{equation*}
	(-K_{Y_{m,d}})^{\text{dim}(Y_{m,d})}=(-K_{Y_{m,d}})^{m+d}=\left( (m+2)%
	\mathfrak{H}\right) ^{m+d}=(m+2)^{m+d}\mathfrak{H}^{m+d}=d(m+2)^{m+d},
\end{equation*}%
Moreover, $K_{Y_{m,d}\times (\mathbb{P}%
	^{1})^{n}}=	\pr_0^*K_{Y_{m,d}}+\sum_{i=1}^{n}\pr_i^*K_{\mathbb{P}%
	^{1}}$ with $-K_{Y_{m,d}}=(m+2)\mathfrak{H}$ and $-K_{\mathbb{P}^{1}}=2%
\mathfrak{h,}$ where $\mathfrak{h}$ is the class of a point on $\mathbb{P}%
^{1}.$ Defining $\mathfrak{H}_{0}:=	\pr_0^*\mathfrak{H}$ and $\mathfrak{%
	H}_{i}:=\pr_i^*\mathfrak{h}$ for all $i\in \{1,...,n\}$ we have $%
\mathfrak{H}_{0}^{m+d+1}=0$, $\mathfrak{H}_{i}^{2}=0,$ $\mathfrak{H}%
_{0}^{m+d}\mathfrak{H}_{1}\cdots \mathfrak{H}_{n}=d$ and {\small
\begin{equation*}
	(-K_{Y_{m,d}\times (\mathbb{P}^{1})^{n}})^{m+d+n}=\left[(m+2)\mathfrak{H}%
	_{0}+\sum_{i=1}^{n}2\mathfrak{H}_{i}\right] ^{m+d+n}=\sum \tfrac{(m+d+n)!}{%
		a_{0}!a_{1}!\cdots a_{n}!}(m+2)^{a_{0}}2^{a_{1}+\cdots a_{n}}\mathfrak{H}%
	_{0}^{a_{0}}\mathfrak{H}_{1}^{a_{1}}\cdots \mathfrak{H}_{n}^{a_{n}},
\end{equation*} }
\noindent where in all summands the exponents $a_{0},a_{1},...,a_{n}$ are integers $%
\geq 0$ with $a_{0}+a_{1}+\cdots +a_{n}=m+d+n.$  Since $\mathfrak{H}_{i}^{2}=0$, any term with $a_{i}\ge 2$ for some $i\in \{1,...,n\}$ vanishes, while $\mathfrak{H}_{0}^{m+d+1}=0$ implies $a_{0}\le m+d$. Hence a non-zero term must satisfy  $a_{0}\le m+d$ and $a_{i}\le 1$ for all $i\in \{1,...,n\}$. In view of the preceding constraint on the sum of the exponents, all these upper bounds must be attained. Thus the \textit{unique surviving} multi-index is $%
(a_{0},a_{1},...,a_{n})=(m+d,1,1,...,1).$  This means that 
\begin{equation*}
	(-K_{Y_{m,d}\times (\mathbb{P}^{1})^{n}})^{m+d+n}=\tfrac{(m+d+n)!}{(m+d)!}%
	(m+2)^{m+d}2^{n}\mathfrak{H}_{0}^{m+d}\mathfrak{H}_{1}\cdots \mathfrak{H}%
	_{n}=2^{n}d(m+2)^{m+d}\tfrac{(m+d+n)!}{(m+d)!},
\end{equation*}%
giving formula \eqref{antican2}.
\end{proof}

\begin{example}
	$(m,d,n)=(2,3,1)$: $(-K_{Y_{2,3}\times \mathbb{P}^{1}})^{6}$ $=2\cdot3\cdot4^5\cdot 6!/5!=2\cdot3\cdot1024\cdot6=36864$.
\end{example}

\begin{remark}\label{rem:rational}
(i) Over the function field \(\CC(\PP^d)\), equation \eqref{eq:fibre-linear} cuts a hyperplane \(\PP^m\) in \(\PP^{m+1}\). Hence $\Yt_{m,d}\bir\PP^m\times\PP^d$, and therefore \(Y_{m,d}\) is \textit{rational} (of dimension $m+d$).  Parameter $d$ controls the transverse multiplicity $d-1$ along $\Lambda$, while $m$ controls $\text{dim}(\Lambda)$. Thus rationality alone imposes no constraint on the pair (dim(Sing), transverse multiplicity) and the family realises many such pairs. \smallskip \newline
\noindent (ii)  $Y_{m,d}\times (\PP^{1})^{n}$ is \textit{rational} as well. Its dimension is $m+d+n$ and, since $(\PP^{1})^{n}$ is smooth,  its singular locus is $\text{Sing}(Y_{m,d}\times (\PP^{1})^{n})=\Lambda \times (\PP^{1})^{n}.$ 
 \smallskip \newline
\noindent (iii) Summing up, we see by using Corollary \ref{cor:prod}, (i) of Theorem \ref{thm:uniform-geometry} and Proposition \ref{PROFANO} that all the members of the (three-parameter) family  $(Y_{m,d}\times (\PP^{1})^{n})_{(m,d,n)}$ (with $m\ge 0, d\ge 3, n\ge 1$) are \textit{projective normal Gorenstein rational Fano varieties with terminal singularities.}   \smallskip \newline
\noindent (iv)  If $m\ge d$, then both $Z_{m,d}$ and $Z'_{m,d}$ are empty for general defining equations and no negativity arises, i.e. $h_{\text{st}}^{p,q}\bigl(Y_{m,d}\times (\PP^{1})^{n}\bigr)\ge 0$ for all $p,q$ and $Y_{m,d}\times (\PP^{1})^{n}$ is not a counterexample to Conjecture \ref{conj:bat}, because \eqref{eq:uniform-Est-raw}, $E(\varnothing;u,v):=0$ and the elementary identity 
\begin{equation}
(1+uv)E(\PP^{m};u,v)-uvE(\PP^{m-1};u,v)=E(\PP^{m+1};u,v) \label{ELIDENT}
\end{equation}
give
  \[\Est \bigl(Y_{m,d}\times (\PP^{1})^{n};u,v\bigr) 
=(1+uv)^{n-1}\bigl[E(\PP^{d};u,v)E(\PP^{m+1};u,v)]. \]
\end{remark}

\section{Master formula and the structure theorem}\label{sec:master}
\noindent  Throughout sections \ref{sec:master}-\ref{Comments} we shall assume that
$
0\le m\le d-1,
$
so that \(Z_{m,d}\) is non-empty. From the equalities \eqref{eq:uniform-Est-raw} and \eqref{ELIDENT}
one obtains%
\begin{align}
        \Est \bigl(Y_{m,d}\times (\PP^{1})^{n};u,v\bigr) \label{Promaster}
        ={}&(1+uv)^{n-1}\bigl[E(\PP^{d};u,v)E(\PP^{m+1};u,v)  \\
        &+(uv)^{m+1}((1+uv)E(Z_{m,d}^{\prime
        };u,v)-E(Z_{m,d};u,v))\bigr]. \nonumber
\end{align}
\begin{remark} If we wish to evaluate numerically the stringy $E$-polynomial \eqref{Promaster} of $%
        Y_{m,d}\times (\PP^{1})^{n},$ then we can use \eqref{eq:EPROJEC} for $\nu =d$ and $\nu =m+1,$
        together with the  lengthy formulas \eqref{eq:E-I} and   \eqref{E-II} which express the $E$-polynomials of the complete intersections $%
        Z_{m,d}$ and $Z_{m,d}^{\prime },$ respectively, by means of the parameters $m$ and $d$. On the other hand, if we are
        content to examine only the sign behaviour of the coefficients of the stringy 
        $E$-polynomial of $Y_{m,d}\times (\PP^{1})^{n}$, then Theorem \ref{thm:structure} and Theorem \ref{thm:classification} are far more convenient than the explicit formulas  \eqref{eq:E-I} and \eqref{E-II}. 
\end{remark}
\noindent Next, we define%
\begin{equation}
        L_{m,d}(u,v):=E(\PP^{d};u,v)E(\PP^{m+1};u,v)+(uv)^{m+1}(E(\PP%
        ^{d-m-2};u,v)-1). \label{eq:Lmd}
\end{equation}

\begin{lemma}\label{NONNEGLMD}
All coefficients of $L_{m,d}(u,v)$ are non-negative.
\end{lemma}
\begin{proof}\textit{First case}: $d\geq m+2.$ If $d=m+2,$ then $E(\PP%
        ^{d-m-2};u,v)-1=E(\PP^{0};u,v)-1=0.$ If  $d\geq m+3,$ then%
        \begin{equation*}
                E(\PP^{d-m-2};u,v)-1=uv+(uv)^{2}+\cdots +(uv)^{d-m-2}
        \end{equation*}%
        which has non-negative coefficients. Since $E(\PP^{d};u,v)E(\PP^{m+1};u,v)$
        also has non-negative coefficients, so does the polynomial $L_{m,d}(u,v).$  \medskip \newline 
\noindent \textit{Second case}: $d=m+1.$ We have $E(\PP^{d-m-2};u,v)=E(\PP^{-1};u,v)=0,
        $ and%
        \begin{equation*}
                L_{m,d}(u,v)=E(\PP^{d};u,v)E(\PP^{m+1};u,v)-(uv)^{m+1}=(E(\PP%
                ^{m+1};u,v))^{2}-(uv)^{m+1}.
        \end{equation*}%
        The coefficient of $(uv)^{m+1}$ in $(E(\PP^{m+1};u,v))^2$ equals the number of
        pairs $(a,b)\in \mathbb{ Z}^2$ satisfying 
        \begin{equation*}
                a+b=m+1,\  \  \ 0\leq a,b\leq m+1.
        \end{equation*}%
        There are exactly $m+2$ such pairs, namely \thinspace $%
        (0,m+1),(1,m),...,(m+1,0).$  After subtracting one copy of $(uv)^{m+1},$ the
        resulting coefficient is $m+1\geq 1.$ Every other coefficient is unchanged
        and non-negative.
\end{proof}

\begin{lemma}[Master formula]
For \(n\ge1\) the $E_{\text{\emph{st}}}$-polynomial of $Y_{m,d}\times (\PP^{1})^{n}$ equals
        \begin{equation}
                \fbox{$%
                        \begin{array}{c}
                                \Est \bigl(Y_{m,d}\times (\PP^{1})^{n};u,v\bigr)=\left( 1+uv\right)
                                ^{n-1}L_{m,d}(u,v) \medskip \\ 
                                -\varepsilon \bigl[\left( 1+uv\right) ^{n}(uv)^{m+1}P_{\text{\emph{prim}}%
                                }(Z_{m,d}^{\prime };u,v)+\left( 1+uv\right) ^{n-1}(uv)^{m+1}P_{\text{\emph{%
                                                        prim}}}(Z_{m,d};u,v)\bigr]\ ,%
                        \end{array}%
                        $}  \label{masterfor}
        \end{equation}
        where
        \begin{equation*}
                P_{\text{\emph{prim}}}(Z_{m,d}^{\prime };u,v)=\sum \limits_{a,b\ge 0,a+b=l}h_{\text{\emph{prim}}}^{a,b}(Z_{m,d}^{\prime })u^{a}v^{b}\text{ \  \emph{and\ }}\ 
                \text{ }P_{\text{\emph{prim}}}(Z_{m,d};u,v)=\sum \limits_{a,b\ge 0,a+b=k}%
                h_{\text{\emph{prim}}}^{a,b}(Z_{m,d})u^{a}v^{b}
        \end{equation*}%
        are the primitive Hodge polynomials \emph{\eqref{PRIMITIVHPOL}} of $V=Z_{m,d}^{\prime }$ and $%
        V=Z_{m,d},$ respectively, and $\varepsilon :=(-1)^{k}.$
\end{lemma}
\begin{proof}
        Applying   \eqref{eq:Epoly} for $V=Z_{m,d}^{\prime }$ and $V=Z_{m,d},$ respectively, we have%
        \begin{equation}
                E(Z_{m,d}^{\prime };u,v)=E(\PP^{l};u,v)-\varepsilon P_{\text{prim}%
                }(Z_{m,d}^{\prime };u,v)  \label{EPOZPRIME}
        \end{equation}%
        and%
        \begin{equation}
                E(Z_{m,d};u,v)=E(\PP^{k};u,v)+\varepsilon P_{\text{prim}}(Z_{m,d};u,v).
                \label{EPOZ}
        \end{equation}%
        Substituting \eqref{EPOZPRIME} and \eqref{EPOZ} into \eqref{Promaster} we get  
        \begin{equation}
                \begin{array}[t]{l}
                        \Est \bigl(Y_{m,d}\times (\PP^{1})^{n};u,v\bigr)  \label{Promaster2} \medskip \\ 
                        =(1+uv)^{n-1}\bigl[E(\PP^{d};u,v)E(\PP^{m+1};u,v)+(uv)^{m+1}((1+uv)E(\PP%
                        ^{l};u,v)-E(\PP^{k};u,v))\bigr] \medskip \\ 
                        -\varepsilon (1+uv)^{n-1}(uv)^{m+1}\bigl[(1+uv)P_{\text{prim}}(Z_{m,d}^{\prime
                        };u,v)+P_{\text{prim}}(Z_{m,d};u,v)\bigr].%
                \end{array}%
        \end{equation}
        Since $\left( 1+uv\right) E(\PP^{l};u,v)-E(\PP^{k};u,v)=E(\PP^{l};u,v)-1$
        with $l=k-1=d-m-2,$ the second line of  \eqref{Promaster2} becomes $(1+uv)^{n-1}L_{m,d}(u,v)$
        and the third line becomes%
        \begin{equation*}
                -\varepsilon \bigl[\left( 1+uv\right) ^{n}(uv)^{m+1}P_{\text{\rm{prim}}%
                }(Z_{m,d}^{\prime };u,v)+\left( 1+uv\right) ^{n-1}(uv)^{m+1}P_{\text{\rm{%
                                        prim}}}(Z_{m,d};u,v)\bigr],
        \end{equation*}%
        giving  \eqref{masterfor}. 
\end{proof}

\begin{theorem}[Structure theorem]\label{thm:structure}
Let \(d\ge3\), \(0\le m\le d-1\), and \(n\ge1\). For all \(p,q\),
\begin{equation}\label{eq:ABC}
        \fbox{$%
                \begin{array}{ccc}
                        & h_{\emph{st}}^{p,q}\bigl(Y_{m,d}\times (\PP^{1})^{n}\bigr)%
                        =A_{p,q}+B_{p,q}-C_{p,q}, & 
                \end{array}%
                $}
\end{equation}%
where the following three quantities\footnote{If $d=m+1,$ then $Z'_{m,d}=\varnothing$. Convention: $h^{a,b}(\varnothing):=0$ and  $h^{a,b}_{\text{prim,mid}}(\varnothing):=0$ for all $a,b$, and  $P_{\text{prim}%
	}(\varnothing;u,v):=0.$\smallskip}\textsuperscript{,}\footnote{Notation: $[u^pv^q]_{\text{coef}}(...)$ means the coefficient of $u^{p}v^{q}$ \textit{after} expanding the polynomial in parentheses into monomials. \textit{Middle primitive Hodge numbers} are defined by \eqref{MIDPRHODGE}.} :
\begin{align}
A_{p,q}&=[u^pv^q]_{\emph{coef}}((1+uv)^{n-1}L_{m,d}(u,v)), \label{eq:A} \\   \noalign{\vskip10pt}
B_{p,q}&=\sum_{j=0}^{n}\tbinom{n}{j}
 h^{p-m-1-j,q-m-1-j}_{\emph{prim,mid}}(Z'_{m,d}),\label{eq:B}\\
C_{p,q}&=\sum_{j=0}^{n-1}\tbinom{n-1}{j}
 h^{p-m-1-j,q-m-1-j}_{\emph{prim,mid}}(Z_{m,d}),\label{eq:C}
\end{align}
 are non-negative. Moreover, $A_{p,q}=0$ whenever \(p\neq q\).
\end{theorem}

\begin{proof} \textit{The first summand.} Both $L_{m,d}(u,v)$ and $(1+uv)^{n-1}L_{m,d}(u,v)$ are polynomials in $uv$
        only. Thus this part is supported on the diagonal $p=q,$ on which the sign
        in $\Est \bigl(Y_{m,d}\times (\PP^{1})^{n};u,v\bigr)$ is $\left( -1\right)
        ^{p+p}=1.$ Consequently, the coefficient of $u^{p}v^{q}$ in this part is
        exactly its contribution to $h_{\text{st}}^{p,p}\bigl(Y_{m,d}\times (\PP%
        ^{1})^{n}\bigr),$ namely $A_{p,p}.$ By Lemma \ref{NONNEGLMD} we see that $A_{p,q}\geq 0$
        in all the cases. \medskip \newline 
$\bullet $ \textit{Expansion of the} $Z_{m,d}^{\prime }$-\textit{primitive
        part}. First, we note that \[P_{\text{prim}}(Z_{m,d}^{\prime };u,v)=\sum \limits_{a,b\ge 0, a+b=l}h_{\text{prim}%
        }^{a,b}(Z_{m,d}^{\prime })u^{a}v^{b}\ \  \text{and} \ \ \left( 1+uv\right) ^{n}=\sum \limits_{j=0}^{n}\tbinom{n}{j}u^{j}v^{j}, \] i.e.
\begin{equation*}
 \left( 1+uv\right) ^{n}(uv)^{m+1}P_{\text{prim}%
        }(Z_{m,d}^{\prime };u,v)=\sum \limits_{j=0}^{n}\sum \limits_{a,b\ge 0,a+b=l}\tbinom{n}{j%
        }h_{\text{prim}}^{a,b}(Z_{m,d}^{\prime })u^{a+m+1+j}v^{b+m+1+j}.
\end{equation*}%
Let us now fix a target monomial $u^{p}v^{q}$. Since \[
(p,q)=(a+m+1+j,b+m+1+j)\Longrightarrow (a,b)=(p-m-1-j,q-m-1-j),\] once $p,q,j$
are fixed, there is only one possible bidegree $(a,b)$ which can contribute.
Hence, the contribution corresponding to that $j$ is $\tbinom{n}{j}h_{\text{%
                prim}}^{p-m-1-j,q-m-1-j}(Z_{m,d}^{\prime }).$ Therefore%
\begin{equation*}
        \lbrack u^{p}v^{q}]_{\text{coef}}\, \bigl(\left( 1+uv\right) ^{n}(uv)^{m+1}P_{%
                \text{prim}}(Z_{m,d}^{\prime };u,v)\bigr)=\sum_{j=0}^{n}\tbinom{n}{j}h_{%
                \text{prim,mid}}^{p-m-1-j,q-m-1-j}(Z_{m,d}^{\prime })=:B_{p,q}.
\end{equation*}%
In formula \eqref{masterfor} $\left( 1+uv\right) ^{n}(uv)^{m+1}P_{\text{prim}%
}(Z_{m,d}^{\prime };u,v)$ appears multiplied by the global sign $%
-\varepsilon .$ The actual coefficient of $u^{p}v^{q}$ in the $E_{\text{st}}$-polynomial \eqref{masterfor}
coming from $Z_{m,d}^{\prime }$ is $-\varepsilon B_{p,q}.$ Since 
\begin{equation*}
        h_{\text{st}}^{p,q}\bigl(Y_{m,d}\times (\PP^{1})^{n}\bigr)=\left( -1\right)
        ^{p+q}[u^{p}v^{q}]_{\text{coef}}\text{\thinspace }\Est \bigl(Y_{m,d}\times (%
        \PP^{1})^{n};u,v\bigr),
\end{equation*}%
the $Z_{m,d}^{\prime }$-primitive contribution to  the stringy Hodge number $h_{\text{st}}^{p,q}\bigl(%
Y_{m,d}\times (\PP^{1})^{n}\bigr)$ is $\left( -1\right) ^{p+q}\left(
-\varepsilon \right) B_{p,q}.$ For every contributing monomial, $\left( -1\right) ^{p+q}=\left(
-1\right) ^{d+m+2j}=\left( -1\right) ^{d+m}$ because 
\begin{equation*}
        p+q=a+b+2(m+1+j)=l+2(m+1+j)=d-m-2+2(m+1+j)=d+m+2j.
\end{equation*}%
Moreover, $\varepsilon =\left( -1\right) ^{k}=\left( -1\right)
^{d-m-1}\Rightarrow -\varepsilon =\left( -1\right) ^{d-m}.$ So \ the $%
Z_{m,d}^{\prime }$-primitive contribution to $h_{\text{st}}^{p,q}$ equals 
\begin{equation*}
        \left( -1\right) ^{p+q}\left( -\varepsilon \right) B_{p,q}=\left( -1\right)
        ^{d+m}\left( -1\right) ^{d-m}B_{p,q}=B_{p,q}\geq 0.
\end{equation*}%
$\bullet $ \textit{Expansion of the} $Z_{m,d}$-\textit{primitive part}. We
note that $   \left( 1+uv\right) ^{n-1}=\sum \limits_{j=0}^{n-1}\tbinom{n-1}{j}%
u^{j}v^{j}$, i.e.
\begin{equation*}
\left( 1+uv\right) ^{n-1}(uv)^{m+1}P_{\text{prim}%
        }(Z_{m,d};u,v)=\sum \limits_{j=0}^{n-1}\sum \limits_{a,b\ge 0, a+b=k}\tbinom{n-1}{j}h_{%
                \text{prim}}^{a,b}(Z_{m,d})u^{a+m+1+j}v^{b+m+1+j}.
\end{equation*}%
Therefore%
\begin{equation*}
        \lbrack u^{p}v^{q}]_{\text{coef}}\, \bigl(\left( 1+uv\right)
        ^{n-1}(uv)^{m+1}P_{\text{prim}}(Z_{m,d};u,v)\bigr)=\sum_{j=0}^{n-1}\tbinom{n-1%
        }{j}h_{\text{prim,mid}}^{p-m-1-j,q-m-1-j}(Z_{m,d})=:C_{p,q}\geq 0.
\end{equation*}%
The $Z_{m,d}$-primitive contribution to the stringy Hodge number $h_{\text{st}}^{p,q}\bigl(%
Y_{m,d}\times (\PP^{1})^{n}\bigr)$ is $\left( -1\right) ^{p+q}\left(
-\varepsilon \right) C_{p,q}.$ Obviously, $\left( -1\right) ^{p+q}=\left(
-1\right) ^{d+m+1+2j}=\left( -1\right) ^{d+m+1}$ because 
\begin{equation*}
        p+q=a+b+2(m+1+j)=k+2(m+1+j)=d-m-1+2(m+1+j)=d+m+1+2j.
\end{equation*}%
Moreover, $\varepsilon =\left( -1\right) ^{k}=\left( -1\right)
^{d-m-1}\Rightarrow -\varepsilon =\left( -1\right) ^{d-m}.$ So \ the $Z_{m,d}
$-primitive contribution to $h_{\text{st}}^{p,q}$ equals 
\begin{equation*}
        \left( -1\right) ^{p+q}\left( -\varepsilon \right) C_{p,q}=\left( -1\right)
        ^{d+m+1}\left( -1\right) ^{d-m}C_{p,q}=-C_{p,q}.
\end{equation*}%
Now  \eqref{eq:A}, \eqref{eq:B} and \eqref{eq:C} imply \eqref{eq:ABC} via \eqref{masterfor}.
\end{proof}

\begin{remark}\label{rem:culprit-correct} Equality \eqref{eq:ABC} shows that negativity can only arise from the third term: if
\(\hst^{p,q}<0\), then
$
C_{p,q}>A_{p,q}+B_{p,q}.
$ The support of $A_{p,q}$  is governed by the diagonal condition $p=q$, whereas parity distinguishes the two primitive contributions: $B_{p,q}\neq 0$ only if $ p+q \equiv d+m (\text{mod}\ 2)$, while $C_{p,q}$ can be non-zero only if $\ p+q \equiv d+m+1 (\text{mod}\ 2)$. Thus, off the diagonal,  $A_{p,q}=0,$  and the parity of $p+q$ determines whether the contribution comes from $B_{p,q}$ or from $-C_{p,q}$. More precisely, on the off-diagonal slice $p\neq q,  \ p+q \equiv d+m+1 (\text{mod}\ 2)$ we get exactly $\hst^{p,q}=-C_{p,q}$ (because $A_{p,q}=B_{p,q}=0$).  On the diagonal, however, $A_{p,p}$ may be positive.  If $m+d$ is odd, then $B_{p,p}=0$ and $\hst^{p,p}= A_{p,p}-C_{p,p}$,  so diagonal negativity occurs precisely when $C_{p,p}>A_{p,p}$.   If $m+d$ is even, then  $C_{p,p}=0$  and $\hst^{p,p}= A_{p,p}+B_{p,p}\ge 0$. (See Propositions \ref{cor:offdiag} and \ref{cor:diagonal-parity} below.)
\end{remark}

\section{Complete negativity analysis}\label{sec:negativity}
\noindent Let  us fix integers $m\ge 0$, \(d\ge3\) and \(n\ge1\) with \(m\le d-1\).
\begin{proposition}[Exact off-diagonal formula]\label{cor:offdiag}
The monomials in the polynomials contributing to $B_{p,q}$ and $C_{p,q}$ have opposite total-degree parity,  while the polynomial defining $A_{p,q}$ is supported on the diagonal. Hence, for \(p\ne q\),
\[
\hst^{p,q}\bigl(Y_{m,d}\times(\PP^1)^n\bigr)=
\begin{cases}
B_{p,q}\ge0,&\emph{if} \ p+q\equiv d+m\  (\emph{mod}\ 2), \medskip\\
-C_{p,q}\le0,&\emph{if} \ p+q\equiv d+m+1\ (\emph{mod}\ 2).
\end{cases}
\]
More explicitly, if \(a+b=k\), \(a\ne b\), and \(0\le j\le n-1\), then
\begin{equation}\label{eq:offdiag-exact}
\hst^{a+m+1+j,b+m+1+j}
\bigl(Y_{m,d}\times(\PP^1)^n\bigr)
=-\tbinom{n-1}{j}h^{a,b}_{\prim}(Z_{m,d}).
\end{equation}
\end{proposition}

\begin{proof}
        The parity statements were already computed in the proof of Theorem \ref{thm:structure}.  A monomial in the polynomial piece defining $B_{p,q}$ has total degree
        $
        d+m+2j\equiv d+m\  (\text{mod}\ 2),
        $
        whereas a monomial in the polynomial piece defining $C_{p,q}$ has total degree
        $
        d+m+1+2j\equiv d+m+1\  (\text{mod}\ 2).
        $
        The two parities are opposite.  Consequently, at a fixed bidegree $(p,q)$,
        $B_{p,q}$ and $C_{p,q}$ can never be simultaneously non-zero. If $p\ne q$, then $A_{p,q}=0$. Therefore Theorem \ref{thm:structure} immediately gives the displayed two-case formula. It remains to prove the sharper equality \eqref{eq:offdiag-exact}.  We fix a
        primitive Hodge type $(a,b)$ of $Z_{m,d}$ with
        $
        a+b=k,
        \  a\ne b,
        $
        and choose $j\in\{0,\ldots,n-1\}$.  Put
        \[
        p':=a+m+1+j,
        \qquad
        q':=b+m+1+j.
        \]
        Since $a\ne b$, we also have $p'\ne q'$, so $A_{p',q'}=0$, and 
        $
        p'+q'=k+2(m+1+j)=d+m+1+2j,
        $
        which belongs to the parity slice $p'+q'\equiv d+m+1\ (\emph{mod}\ 2)$, so $B_{p',q'}=0$. Now we inspect $C_{p',q'}$:
\[
                C_{p',q'}=\sum_{r\ge0}\tbinom{n-1}{r} h^{p'-m-1-r,q'-m-1-r}_{\text{prim,mid}}(Z_{m,d})=\sum_{r\ge0}\tbinom{n-1}{r} h^{a+j-r,b+j-r}_{\text{prim,mid}}(Z_{m,d}).
\]
A middle primitive Hodge number of $Z_{m,d}$ can be non-zero only when the two indices sum to $k$.  Since
        \[
        (a+j-r)+(b+j-r)=k+2(j-r),
        \]
     this occurs if and only if $r=j$.  Thus every summand vanishes except the one with $r=j$, and
$
        C_{p',q'}=\tbinom{n-1}{j}h^{a,b}_{\prim}(Z_{m,d}).
$
        Since at $(p',q')$ we have $\hst^{p',q'}=-C_{p',q'}$, formula
\eqref{eq:offdiag-exact} follows.
\end{proof}

\begin{remark}
For \(n=1\) formula \eqref{eq:offdiag-exact} becomes especially simple. Since $j=0$ is the only possible shift, one gets $\hst^{p+m+1,q+m+1}
\bigl(Y_{m,d}\times \PP^1\bigr)
=-h^{p,q}_{\prim}(Z_{m,d})$ for the relevant off-diagonal primitive bidegrees $(p,q)$.
\end{remark}
\begin{proposition}[Parity criterion for diagonal negativity]\label{cor:diagonal-parity}
The polynomial $P_{\emph{prim}}(Z_{m,d};u,v)$ can have diagonal support only if
\(k=d-m-1\) is even, i.e. only if \(m+d\) is odd. Consequently:
\begin{enumerate}[label=(\roman*)]
\item if \(m+d\) is even, then
\begin{equation}
\hst^{p,p}\bigl(Y_{m,d}\times(\PP^1)^n\bigr)=A_{p,p}+B_{p,p}\ge0 \label{SYNTH1}
\end{equation}
for every \(p\); diagonal negativity is impossible;
\item if \(m+d\) is odd, then 
\begin{equation}
\hst^{p,p}\bigl(Y_{m,d}\times(\PP^1)^n\bigr)=A_{p,p}-C_{p,p}, \label{SYNTH2}
\end{equation}
and diagonal negativity is possible exactly when $C_{p,p}> A_{p,p}$.
\end{enumerate}
\end{proposition}

\begin{proof} Every monomial of $P_{\text{prim}}(Z_{m,d};u,v)$ has bidegree $(a,b)$ with
        $
        a+b=k.
        $
        Such a monomial is diagonal precisely when $a=b$.  If $a=b$, then
$
        2a=k,
$
        so $k$ must be even.  Conversely, when $k$ is even, the only possible diagonal middle bidegree is
$
        \left(\frac{k}{2},\frac{k}{2}\right).
$ Thus $P_{\text{prim}}(Z_{m,d};u,v)$ can meet the diagonal only if $k$ is even.
Now $ k=d-m-1.$  Since
$
        (d+m)-(d-m)=2m,
$
        \[
        k\text{ even}
        \iff d-m\text{ odd}
        \iff d+m\text{ odd}.
        \]
For $Z'_{m,d}$ the middle dimension is $l=k-1$.  Hence the parity condition is complementary: $P_{\text{prim}}(Z'_{m,d};u,v)$  can have diagonal support only if $l$ is even, i.e. only if $k$ is odd.
If $m+d$ is even, then $k$ is odd.  Therefore $C_{p,p}=0$ for all $p$, while $B_{p,p}$ may occur.  The structure theorem gives \eqref{SYNTH1}. If $m+d$ is odd, then $k$ is even and $l$ is odd.  Therefore
$B_{p,p}=0$ for all $p$, while $C_{p,p}$ may occur, leading to \eqref{SYNTH2}. This quantity is negative exactly when $C_{p,p}>A_{p,p}$.  That is the claimed parity
dichotomy.
\end{proof}

\begin{proposition}[The unique diagonal candidate for \(n=1\)]\label{thm:diagonal-candidate}
Assume \(m+d\) is odd, \(0\le m\le d-1\), and \(n=1\). Then there is at most one diagonal entry which can be negative, namely
\[
p=q=p_0:=\frac{m+d+1}{2}.
\]
\emph{(i)} If \(d\ge m+3\) and
\(\mathfrak{k}:=(d-m-1)/2=k/2\), then
\begin{equation}\label{eq:diag-formula-positive-dim}
\hst^{p_0,p_0}(Y_{m,d}\times\PP^1)
=(m+3)-h^{\mathfrak{k},\mathfrak{k}}_{\prim}(Z_{m,d}).
\end{equation}
\emph{(ii)} If \(d=m+1\), then \(Z_{m,d}\) consists of $(d-1)^d$ reduced points and
\begin{equation}\label{eq:diag-formula-zero-dim}
\hst^{m+1,m+1}(Y_{m,d}\times\PP^1)
=(d+1)-(d-1)^d.
\end{equation}
\end{proposition}

\begin{proof}
        Because $m+d$ is odd, \eqref{SYNTH2} gives $\hst^{p,p}\bigl(Y_{m,d}\times\PP^1\bigr)=A_{p,p}-C_{p,p}. $ Since  $n=1$, formula \eqref{eq:C} simplifies to
        \[
        C_{p,p}=\sum_{j\ge0}\tbinom{0}{j}
        h^{p-m-1-j,p-m-1-j}_{\text{prim,mid}}(Z_{m,d})
        =h^{p-m-1,p-m-1}_{\text{prim,mid}}(Z_{m,d}).
        \]
        By definition, a middle primitive Hodge number of $Z_{m,d}$ can be non-zero only in total degree $k$.
        Hence a diagonal middle primitive term can occur only if
$
        2(p-m-1)=k=d-m-1,
$
i.e., only if $p=\frac{m+d+1}{2}=p_0.$ Thus $C_{p,p}=0$ for every $p\ne p_0$, and all those diagonal entries equal
        $A_{p,p}\ge0$.  This proves that $p_0$ is the unique \emph{candidate} for
        diagonal negativity.
\medskip \newline
        \noindent (i) If $d\ge m+3$, then 
$
        p_0=m+1+\mathfrak{k}
$
and     
$
        A_{p_0,p_0}=[(uv)^{p_0}]_{\text{coef}}\,(L_{m,d}(u,v)).
$
We shall compute the coefficient in the two summands of \eqref{eq:Lmd}.
We first consider $E(\PP^{d};u,v)E(\PP^{m+1};u,v)$.  The coefficient of $(uv)^{p_0}$ is the number of
        pairs $(a,b)$ such that
$a+b=p_0,\  0\le a\le d,\ 0\le b\le m+1.
$ For every $b=0,1,\ldots,m+1$, we put $a=p_0-b$.  We must check that this $a$ lies
        between $0$ and $d$.  Its minimum occurs for $b=m+1$ and equals
$
        p_0-(m+1)=\mathfrak{k} \ge1.
$ Its maximum occurs for $b=0$ and equals $p_0$.  Since $m+1\le d$, we have 
$
        p_0=\frac{m+d+1}{2}\le d.
$ Thus all $m+2$ choices of $b$ are valid, and
        \[
        [(uv)^{p_0}]_{\text{coef}}\,(E(\PP^{d};u,v)E(\PP^{m+1};u,v))=m+2.
        \]
        Next, we consider
        \[
        (uv)^{m+1}(E(\PP^{d-m-2};u,v)-1)
        =(uv)^{m+1}(uv+(uv)^2+\cdots+(uv)^{d-m-2}).
        \]
        If we consider a term $(uv)^r$ in the parentheses (where $r\in \{1,...,d-m-2\}$),  then multiplying by $(uv)^{m+1}$ gives $(uv)^{m+1+r}$. We are specifically looking for the coefficient of $(uv)^{p_0}$. Therefore the total exponent equals $p_0$, i.e., $m+1+r=p_0$ or, equivalently, $r=p_0-(m+1)$. But $p_0-(m+1)=\mathfrak{k}$. Hence, to obtain the term   $(uv)^{p_0}$, we must take $r=\mathfrak{k}$. It remains to verify that this term actually occurs. The available exponents are $1, 2, ...,d-m-2$. Thus we need $1\le\mathfrak{k}\le d-m-2$. Under the assumption $d \ge m+3$, this condition holds. Therefore there is exactly \textit{one} such term, with coefficient $1$, and hence the second summand of $L_{m,d}(u,v)$ contributes exactly $1$ to the coefficient of  $(uv)^{p_0}$. This means that 
        $
        A_{p_0,p_0}=m+2+1=m+3 $ and $C_{p_0,p_0}=h^{\mathfrak{k},\mathfrak{k}}_{\prim}(Z_{m,d}),$
leading to \eqref{eq:diag-formula-positive-dim}.
        
        \medskip
        \noindent (ii) If $d=m+1$, then $k=0$, so $Z_{m,d}$ is a zero-dimensional complete intersection in $\PP^d$ of
        $m+1=d$ hypersurfaces, each of which has degree $d-1$.  For a general choice the
        intersection is reduced and transverse, and B\'{e}zout's theorem gives
        $(d-1)^{m+1}=(d-1)^d$
        points. Now $p_0=m+1=d$.  Since 
        \[      L_{m,d}(u,v)=E(\PP^{d};u,v)E(\PP^{m+1};u,v)-(uv)^{m+1}=(E(\PP%
        ^{m+1};u,v))^{2}-(uv)^{m+1}.\]
        As in the second case of the proof of Lemma \ref{NONNEGLMD}, the coefficient of $(uv)^{m+1}$ in
        $(E(\PP%
        ^{m+1};u,v))^{2}$ is $m+2$, so by \eqref{eq:A} 
        $
        A_{m+1,m+1}=m+1.
        $
        For a set of $(d-1)^d$ points we have 
        \[
         P_{\text{prim}}(Z_{m,d};u,v)=(d-1)^d-1 \Longrightarrow         C_{m+1,m+1}=(d-1)^d-1 .
        \]
        Applying \eqref{SYNTH2} we get \eqref{eq:diag-formula-zero-dim}.
\end{proof}

\begin{remark}If $(m,d)=(0,3)$, then $k=2$, $\mathfrak{k}=1$, and $p_0=2$.  Here $Z_{0,3}$ is a
        smooth quadric surface in $\PP^3$, hence $Z_{0,3}\cong\PP^1\times\PP^1$ (cf. \cite[Ch. I, Ex. 2.15, p. 13]{Har77}) and
        $
        h^{1,1}(Z_{0,3})=2.
        $
        The ambient hyperplane class contributes one dimension, so
$h^{1,1}_{\mathrm{prim}}(Z_{0,3})=1.$
        Formula \eqref{eq:diag-formula-positive-dim} therefore gives
       $
        \hst^{2,2}(Y_{0,3}\times\PP^1)=3-1=2>0.
        $
        This shows that the word \textquotedblleft candidate\textquotedblright  \ in Proposition \ref{thm:diagonal-candidate} is essential: for $(m,d)=(0,3)$ the unique diagonal candidate is in fact positive.
\end{remark}
\medskip
\noindent \textbf{Proof of Theorem \ref{thm:classification}.}  In Sections \ref{sec:one-divisor} and \ref{sec:geometry}, we showed that the varieties
under consideration satisfy the hypotheses of Conjecture \ref{conj:bat}:
$Y_{m,d}$ is projective and Gorenstein, the blow-up along its singular linear
space gives a log resolution with discrepancy $1$, hence $Y_{m,d}$ is
terminal (in particular canonical), and after multiplying by $(\PP^1)^n$ the
stringy $E$-function is a polynomial for every $n\ge1$.  Thus, to prove that a
member of the family $Y_{m,d}\times(\PP^1) ^{n}$ is a counterexample for $d\ge 3$ and $0\le m\le d-1$, it is enough to exhibit at least one negative stringy Hodge
number. We distinguish cases according to the value of
\[
k=\dim (Z_{m,d})=d-m-1.
\]

\subsection*{\noindent $\bullet$ Case 1: $k\ge1$}
We prove that, except for $(m,d)=(0,3)$, the smooth complete intersection $Z_{m,d}$
has a non-zero \emph{off-diagonal} primitive Hodge number.  Formula \eqref{eq:offdiag-exact} of Proposition \ref{cor:offdiag} then
immediately produces a negative stringy Hodge number for every $n\ge1$.

\medskip
\noindent\textbf{Subcase 1A: $k=1$.}
Then
$
m=d-2,
$ $Z_{m,d} =Z_{d-2,d}$ is a smooth curve in $\PP^d$ and formula \eqref{hpqprimZMD} simplifies to 
\begin{equation*}
	h_{\text{prim}}^{1,0}(Z_{m,d})=h_{\text{prim}}^{0,1}(Z_{m,d})=%
	\sum_{j=1}^{d-1}(-1)^{d-1+j}\tbinom{d-1}{j}\tbinom{(d-1)j-1}{d}=1+\frac{%
		d(d-3)}{2}(d-1)^{d-1}>0.
\end{equation*}
\medskip
\noindent\textbf{Subcase 1B: $k\ge2$ and $m\ge1$.} Adjunction formula gives
\begin{equation*}
        K_{Z_{m,d}}\cong \cO_{Z_{m,d}}\left((m+1)(d-1)-(d+1)\right)=\cO_{Z_{m,d}}\left(m(d-1)-2\right).
\end{equation*}
Since $k=d-m-1\ge2$, we have $d\ge m+3$.  Hence, for $m\ge1$,
$
m(d-1)-2\ge m(m+2)-2\ge1.
$
We define $\mathfrak{d}:=m(d-1)-2\ge1$ and choose a hyperplane $\mathcal H$ in $\PP^d$ not containing $Z_{m,d}$.  The restriction $\left. \mathbb{L}\right \vert _{Z_{m,d}}$ of the (non-zero) defining linear form $\mathbb{L}$  of  $\mathcal H$ to $Z_{m,d}$ is not identically zero. $\mathbb{L}$ is a global section of $\mathcal{O}_{\PP^{d}}(1),$ so after restriction $%
0\neq \left. \mathbb{L}\right \vert _{Z_{m,d}}\in H^{0}(Z_{m,d},\mathcal{O}%
_{Z_{m,d}}(1)).$ Since $\mathfrak{d}\geq 1,$ its $\mathfrak{d}$-th power $%
(\left. \mathbb{L}\right \vert _{Z_{m,d}})^{\mathfrak{d}}$ is a section of $%
\mathcal{O}_{Z_{m,d}}(1)^{\otimes \mathfrak{d}}=\mathcal{O}_{Z_{m,d}}(%
\mathfrak{d}).$ Moreover, there is an open set on which $\left. \mathbb{L}%
\right \vert _{Z_{m,d}}$ does not vanish. Hence its $\mathfrak{d}$-th
power $(\left. \mathbb{L}\right \vert _{Z_{m,d}})^{\mathfrak{d}}$ cannot
vanish identically there, i.e.,   
\begin{equation*}
	\left. 
	\begin{array}{r}
		0\neq (\left. \mathbb{L}\right \vert _{Z_{m,d}})^{\mathfrak{d}}\in
		H^{0}(Z_{m,d},\mathcal{O}_{Z_{m,d}}(\mathfrak{d}))\medskip  \\ 
		K_{Z_{m,d}}\cong \mathcal{O}_{Z_{m,d}}(\mathfrak{d})%
	\end{array}%
	\right \} \Longrightarrow H^{0}(Z_{m,d},K_{Z_{m,d}})\neq 0\Longrightarrow
	h^{0}(Z_{m,d},K_{Z_{m,d}})>0,
\end{equation*}%
and therefore 
\begin{equation*}
h^{0}(Z_{m,d},K_{Z_{m,d}})=h^{0}(Z_{m,d},\Omega
	_{Z_{m,d}}^{k})=h^{k,0}(Z_{m,d})=h_{\text{prim}}^{k,0}(Z_{m,d})>0
\end{equation*}%
with $k\geq 2$. For all $j\in \{0,...,n-1\}$ \eqref{eq:offdiag-exact}  gives 
\begin{equation*}
	h_{\text{st}}^{d+j,m+1+j}(Y_{m,d}\times (\PP^{1})^{n})=-\tbinom{n-1}{j}h_{%
		\text{prim}}^{k,0}(Z_{m,d})<0.
\end{equation*}%

\medskip
\noindent\textbf{Subcase 1C: $m=0$ and $d\ge4$.} Now  $Z_{m,d} =Z_{0,d}\subset\PP^d$  is a smooth hypersurface of degree $d-1$ and dimension
$d-1$.  Formula \eqref{hpqprimZMD} gives 
\begin{equation*}
	h_{\text{prim}}^{p,k-p}(Z_{0,d})=h_{\text{prim}}^{p,d-1-p}(Z_{0,d})=%
	\sum_{i=0}^{p}(-1)^{i}\tbinom{d+1}{i}\tbinom{(d-1)(p+1)-(d-2)i-1}{d}.
\end{equation*}%
In particular, for $p=1$ we get 
$
	h_{\text{prim}}^{d-2,1}(Z_{0,d})=h_{\text{prim}}^{1,d-2}(Z_{0,d})=\tbinom{2d-3%
	}{d}>0.
$
This is off-diagonal, so equality  \eqref{eq:offdiag-exact}  of Proposition \ref{cor:offdiag} again gives negativity for every
$n\ge1$.\medskip

\noindent\textbf{Subcase 1D: $m=0$ and $d=3$.} This is the only positive-dimensional case not covered by Subcases 1A--1C.  Using \eqref{hpqprimZMD}, since $m+1=1,$ the $j$-sum contains only the term $j=1.$ Thus%
\begin{equation*}
	h_{\text{prim}}^{p,2-p}(Z_{0,3})=\sum_{i=0}^{p}(-1)^{i}\tbinom{4}{i}\tbinom{%
		2p+1-i}{3}\Longrightarrow h_{\text{prim}}^{0,2}(Z_{0,3})=0,\ h_{\text{prim}%
	}^{1,1}(Z_{0,3})=1,\ h_{\text{prim}}^{2,0}(Z_{0,3})=0
\end{equation*}%
(for $p=0,1,2$). This means that $P_{\text{prim}}(Z_{0,3};u,v)=uv$ and by \eqref{eq:Epoly},%
\begin{equation}
	E(Z_{0,3};u,v)=(1+uv+(uv)^{2})+P_{\text{prim}%
	}(Z_{0,3};u,v)=1+2uv+(uv)^{2}=(1+uv)^{2}.\label{EZ03}
\end{equation}%
On the other hand, formula \eqref{E-II} (in which $l=1$, the $r=0$ contribution is $-4u$ and  the $r=1$ contribution is $-4v$) gives
\begin{equation}
	E(Z_{0,3}^{\prime };u,v)=1-4u-4v+uv. \label{EPRZ03}
\end{equation}%
Applying \eqref{eq:EPROJEC} for $\nu =1,3$ and combining \eqref{Promaster} with \eqref{EZ03} and \eqref{EPRZ03} we obtain 
\begin{align}
	\Est \bigl(Y_{0,3}\times (\PP^{1})^{n};u,v\bigr)={}& (1+uv)^{n-1}\bigl[E(\PP%
	^{3};u,v)E(\PP^{1};u,v)   \nonumber \\
	& +(uv)((1+uv)E(Z_{0,3}^{\prime };u,v)-E(Z_{0,3};u,v))\bigr]  \notag \\
	& ={}\left( uv+1\right) ^{n}(1+uv-4uv^{2}-4u^{2}v+u^{2}v^{2}+u^{3}v^{3}). \label{Promaster3}
\end{align}%
$\allowbreak $Expanding $\left( uv+1\right) ^{n}=\sum_{j=0}^{n}\tbinom{n}{j}%
u^{j}v^{j},$ we conclude  via  \eqref{Promaster3}  that {\small
\begin{equation*}
	\Est \bigl(Y_{0,3}\times (\PP^{1})^{n};u,v\bigr)=\sum_{j=0}^{n}\tbinom{n}{j}%
	(u^{j}v^{j}+u^{j+1}v^{j+1}+u^{j+2}v^{j+2}+u^{j+3}v^{j+3})-4\sum_{j=0}^{n}%
	\tbinom{n}{j}(u^{j+2}v^{j+1}+u^{j+1}v^{j+2}).
\end{equation*}}
\noindent Therefore for every $j\in \{0,...,n\},$%
\begin{equation*}
	\lbrack u^{j+2}v^{j+1}]_{\text{coef}}\, \Est \bigl(Y_{0,3}\times (\PP%
	^{1})^{n};u,v\bigr)=[u^{j+1}v^{j+2}]_{\text{coef}}\, \Est \bigl(Y_{0,3}\times (%
	\PP^{1})^{n};u,v\bigr)=-4\tbinom{n}{j}.
\end{equation*}%
But these monomials have \textit{odd }total degree $(j+2)+(j+1)=2j+3$ and $%
E_{\text{st}}=\sum_{p,q}(-1)^{p+q}h_{\text{st}}^{p,q}u^{p}v^{q}.$ Hence 
\begin{equation*}
	h_{\text{st}}^{j+2,j+1}(Y_{0,3}\times (\PP^{1})^{n})=h_{\text{st}%
	}^{j+1,j+2}(Y_{0,3}\times (\PP^{1})^{n})=4\tbinom{n}{j}>0 \ \ (\text{for}\ 0\le j \le n)
\end{equation*}%
and along the diagonal 
\begin{equation*}
	h_{\text{st}}^{r,r}(Y_{0,3}\times (\PP^{1})^{n})=\tbinom{n}{r}+\tbinom{n}{r-1%
	}+\tbinom{n}{r-2}+\tbinom{n}{r-3}>0\text{ \  \ (whenever }0\leq r\leq n+3\text{)}
\end{equation*}%
i.e., $h_{\text{st}}^{p,q}(Y_{0,3}\times (\PP^{1})^{n})\geq 0$ for all $%
p,q$ and all $n\geq 1,$ proving exception (i). 

\subsection*{$\bullet$ Case 2: $k=0$}
Now
$
d=m+1, \ \text{i.e.}, m=d-1.
$
The variety $Z^{\prime}_{m,d}$ has expected dimension $-1$, hence $Z^{\prime}_{m,d}=\varnothing$.  The
complete intersection $Z_{m,d}$ is cut out in $\PP^d$ by $d$ general hypersurfaces of
degree $d-1$.  By B\'{e}zout and transversality it consists of
$ (d-1)^d$ reduced points.  Since $E(Z_{d-1,d};u,v)=(d-1)^{d}=$ $E(\mathbb{P}^{0};u,v)+P_{\text{prim}%
}(Z_{d-1,d};u,v),$ i.e.,%
\begin{equation*}
	P_{\text{prim}}(Z_{d-1,d};u,v)=h_{\text{prim}}^{0,0}(Z_{d-1,d})=(d-1)^{d}-1.
\end{equation*}%
This is exactly what \eqref{hpqprimZMD} gives: 
\begin{equation*}
	h_{\text{prim}}^{0,0}(Z_{d-1,d})=\sum_{j=1}^{d}(-1)^{d+j}\tbinom{d}{j}%
	\tbinom{(d-1)j-1}{d}=(d-1)^{d}-1.
\end{equation*}%
From \eqref{eq:Lmd}, $L_{d-1,d}(u,v)=(1+uv+\cdots +(uv)^{d})^{2}-(uv)^{d},$ and master
formula \eqref{masterfor} becomes%
\begin{equation}
	\Est \bigl(Y_{d-1,d}\times (\PP^{1})^{n};u,v\bigr)=(1+uv)^{n-1}[(1+uv+\cdots
	+(uv)^{d})^{2}-(d-1)^{d}(uv)^{d}]. \label{ESTD-1D}
\end{equation}%
This is a polynomial in $uv$ alone. Hence every coefficient is a stringy Hodge
number, because the sign $(-1)^{2p}$ is $+1.$ 

\medskip
\noindent\textbf{Subcase 2A: $d=3$ (thus $m=2$).}
Here $(d-1)^{d}=2^3=8$, $(E(\PP^{3};u,v))^2=(1+uv+(uv)^2+(uv)^3)^2$. The coefficient vector of  $(E(\PP^{3};u,v))^2$ is $(1,2,3,4,3,2,1).$ Subtracting $8(uv)^3$ gives, for $n=1$,  $(1,2,3,-4,3,2,1).$  For each increase of $n$ by $1$, formula \eqref{ESTD-1D} multiplies the polynomial by $1+uv$, which replaces a coefficient vector $(a_0,\ldots,a_r)$ by
$
(a_0,a_0+a_1,a_1+a_2,\ldots,a_{r-1}+a_r,a_r).
$
Thus the coefficient vectors for $n=1,2,3,4$ are respectively
\begin{align*}
        &(1,2,3,-4,3,2,1),\\
        &(1,3,5,-1,-1,5,3,1),\\
        &(1,4,8,4,-2,4,8,4,1),\\
        &(1,5,12,12,2,2,12,12,5,1).
\end{align*}
Therefore negative diagonal stringy Hodge numbers occur for $n=1,2,3$, while
for $n=4$ every coefficient is non-negative. For every $n\ge4$ we can factor
\[
	\Est \bigl(Y_{2,3}\times (\PP^{1})^{n};u,v\bigr)=(1+uv)^{n-4}\Est \bigl(Y_{2,3}\times (\PP^{1})^{4};u,v\bigr).
\]
Both factors have non-negative coefficients, so their product has
non-negative coefficients.  Hence no negative stringy Hodge number occurs for
$n\ge4$.  This proves exception (ii) and the asserted negativity for
$n=1,2,3$.

\medskip
\noindent\textbf{Subcase 2B: $d\ge4$.}
Evaluating \eqref{ESTD-1D}  at $u=v=1$ we get
\begin{equation}\label{eq:w1}
       	\Est \bigl(Y_{d-1,d}\times (\PP^{1})^{n};1,1\bigr)=2^{n-1}\left((d+1)^2-(d-1)^d\right).
\end{equation}
Since $d\ge4$ we have
$
(d-1)^{d-2}\ge(d-1)^2>d+1 \Longrightarrow (d-1)^d=(d-1)^{d-2}(d-1)^2>(d+1)^2.
$
Thus the right side of \eqref{eq:w1} is strictly negative. But the polynomial in \eqref{ESTD-1D} is supported entirely on the diagonal.  Write it as
\[
	\Est \bigl(Y_{d-1,d}\times (\PP^{1})^{n};u,v\bigr)= \sum_r a_r(uv)^r.
\]
Since $(-1)^{r+r}=1$, we have $a_r=\hst^{r,r}\bigl(Y_{d-1,d}\times (\PP^{1})^{n}\bigr)$.  If every $a_r$ were non-negative, then
\[
\Est \bigl(Y_{d-1,d}\times (\PP^{1})^{n};1,1\bigr)=\sum_r a_r\ge0,
\]
contradicting \eqref{eq:w1}.  Hence at least one coefficient $a_r$, equivalently
at least one diagonal stringy Hodge number, is negative for every $n\ge1$.
This proves the final assertion and completes the classification.\hfill $\Box$
\begin{corollary} For every $\mathfrak{d}\ge 5$ there is a $\mathfrak{d}$-dimensional projective variety with Gorenstein terminal singularities, a polynomial stringy $E$-function and a negative stringy Hodge number, namely  $Y_{1,3}\times (\PP^{1})^{\mathfrak{d}-4}$. Explicitly, 
	\begin{equation}
	\hst^{3+j,2+j}(Y_{1,3}\times (\PP^{1})^{\mathfrak{d}-4})=\hst^{2+j,3+j}(Y_{1,3}\times (\PP^{1})^{\mathfrak{d}-4})=-\tbinom{\mathfrak{d}-5}{j}, \qquad 0\le j\le \mathfrak{d}-5. \label{HSEXAMP}
	\end{equation}
\end{corollary}
\noindent This recovers Huang$-$Satriano’s counterexamples in all dimensions $\ge 5$ within a single family. See \S \ref{HSMEMBER} below.

\section{Examples}\label{sec:special}

\subsection{The Huang--Satriano member \((m,d)=(1,3)\)}\label{HSMEMBER}
Here \(Z_{1,3}\subset\PP^3\) is a \((2,2)\) elliptic curve. Thus
$h^{1,0}_{\prim}(Z_{1,3})=h^{0,1}_{\prim}(Z_{1,3})=1$, and Proposition \ref{cor:offdiag} gives
\[
\hst^{3+j,2+j}(Y_{1,3}\times (\PP^{1})^{n})=\hst^{2+j,3+j}(Y_{1,3}\times (\PP^{1})^{n})=-\tbinom{n-1}{j}, \qquad 0\le j\le n-1.
\]
For \(n=1\), this is \(-1\). Setting $n=\mathfrak{d}-4$ we get \eqref{HSEXAMP}. Since \(m+d=4\) is even,  Proposition \ref{cor:diagonal-parity} explains structurally why no diagonal negativity occurs. In fact,%
\begin{equation*}
	h_{\text{st}}^{r,r}(Y_{1,3}\times (\PP^{1})^{n})=\sum_{i=0}^{5}\lambda _{i}%
	\tbinom{n-1}{r-i},\text{ \  \ with\  \ }\left( \lambda _{0},\lambda
	_{1},\lambda _{2},\lambda _{3},\lambda _{4},\lambda _{5}\right)
	=(1,2,14,14,2,1).
\end{equation*}%
\subsection{The member with parameters \((m,d)=(0,4)\)}
Here \(Z_{0,4}\subset\PP^4\) is a smooth cubic threefold, with $h^{2,1}_{\prim}(Z_{0,4})=h^{1,2}_{\prim}(Z_{0,4})=5$.  The shift is \((1,1)\), so
\begin{equation*}\label{eq:familyA-negative}
\hst^{3+j,2+j}(Y_{0,4}\times (\PP^{1})^{n})=\hst^{2+j,3+j}(Y_{0,4}\times (\PP^{1})^{n})=-5\tbinom{n-1}{j},
\qquad 0\le j\le n-1.
\end{equation*}
Again \(m+d=4\) is even, hence every diagonal stringy Hodge number is non-negative. On the other hand, the $Z'_{0,4}$-primitive terms have even total degree after the shift, and they contribute positively:
\begin{equation*}\label{eq:familyA-pos}
	\hst^{3+j,1+j}(Y_{0,4}\times (\PP^{1})^{n})=\hst^{1+j,3+j}(Y_{0,4}\times (\PP^{1})^{n})=15\tbinom{n}{j},
	\qquad 0\le j\le n.
\end{equation*}
\subsection{The member with parameters \((m,d)=(1,4)\)}
Now \(Z_{1,4}\subset\PP^4\) is a smooth \((3,3)\) surface with $h^{2,0}_{\prim}(Z_{1,4})=h^{0,2}_{\prim}(Z_{1,4})=5$, and  $h^{1,1}(Z_{1,4})=51,\ h^{1,1}_{\prim}(Z_{1,4})=50.$ The $Z'_{1,4}$-contribution has negative coefficients but odd total degree, so it gives positive stringy Hodge numbers:
\begin{equation*}\label{eq:familyB-pos}
	\hst^{3+j,2+j}(Y_{1,4}\times (\PP^{1})^{n})=\hst^{2+j,3+j}(Y_{1,4}\times (\PP^{1})^{n})=91\tbinom{n}{j},
	\qquad 0\le j\le n.
\end{equation*}
By contrast, the $Z_{1,4}$-primitive off-diagonal terms have even total degree and give
\begin{equation*}\label{eq:familyB-neg}
	\hst^{4+j,2+j}(Y_{1,4}\times (\PP^{1})^{n})=\hst^{2+j,4+j}(Y_{1,4}\times (\PP^{1})^{n})=-5\tbinom{n-1}{j},
	\qquad 0\le j\le n-1.
\end{equation*}
For $n = 1$ the negative stringy Hodge numbers are
\[
	\hst^{3,3}(Y_{1,4}\times \PP^{1})=-46 \ \ \text{and}\ \  	\hst^{4,2}(Y_{1,4}\times \PP^{1})=\hst^{2,4}(Y_{1,4}\times \PP^{1})=-5 .
\]
This is the first of the two members in \S \ref{COMPAR}  with $m+d$ odd for which diagonal and off-diagonal negativity coexist.

\subsection{The member with parameters \((m,d)=(0,5)\).} Here \(Z_{0,5}\subset\PP^5\) is a smooth quartic fourfold, with {\small
\begin{equation*}
	h_{\text{prim}}^{0,4}(Z_{0,5})=h_{\text{prim}}^{4,0}(Z_{0,5})=0,\  \ h_{\text{%
			prim}}^{1,3}(Z_{0,5})=h_{\text{prim}}^{3,1}(Z_{0,5})=21,\  \
	h^{2,2}(Z_{0,5})=142,\  \ h_{\text{prim}}^{2,2}(Z_{0,5})=141.
\end{equation*}}%
The $Z'_{0,5}$ terms have negative coefficients but odd total degree and therefore produce positive
stringy Hodge numbers:
\begin{equation*}
	\hst^{3+j,2+j}(Y_{0,5}\times (\PP^{1})^{n})=\hst^{2+j,3+j}(Y_{0,5}\times (\PP^{1})^{n})=1126\tbinom{n}{j},
	\qquad 0\le j\le n,
\end{equation*}
and 
\begin{equation*}
	\hst^{4+j,1+j}(Y_{0,5}\times (\PP^{1})^{n})=\hst^{1+j,4+j}(Y_{0,5}\times (\PP^{1})^{n})=56\tbinom{n}{j},
	\qquad 0\le j\le n.
\end{equation*}
The primitive $(3, 1)$ and $(1, 3)$ classes of $Z_{0,5}$ produce the negative off-diagonal stringy Hodge
numbers
\begin{equation*}
	\hst^{4+j,2+j}(Y_{0,5}\times (\PP^{1})^{n})=\hst^{2+j,4+j}(Y_{0,5}\times (\PP^{1})^{n})=-21\tbinom{n-1}{j},
	\qquad 0\le j\le n-1.
\end{equation*}
For $n = 1$, the only negative stringy Hodge numbers are
\begin{equation*}
	\hst^{3,3}(Y_{0,5}\times \PP^{1})=-138,\ \  	\hst^{4,2}(Y_{0,5}\times \PP^{1})=\hst^{2,4}(Y_{0,5}\times \PP^{1})=-21.
\end{equation*}
The equality $3 -141=-138 $ shows directly the diagonal mechanism: a small ambient contribution is overwhelmed by the
primitive middle Hodge number $141$ of the quartic fourfold.
\subsection{Comparison of the above members}\label{COMPAR}
The essential data may be summarised as follows:\vspace{0.3cm}
\renewcommand {\arraystretch}{1.4}
\begin{center}
	\small
	\begin{tabular}{@{}c c c p{5.0cm} p{5.3cm}@{}}
		\toprule
		$(m,d)$ & $k$ & parity of $m+d$ & Key primitive data of $Z_{m,d}$ & Negative $\hst^{p,q}$'s of $Y_{m,d}\times \PP^{1}$ \\
		\midrule
		$(1,3)$ & $1$ & even & $\hprim^{1,0}=\hprim^{0,1}=1$ & $\hst^{3,2}=\hst^{2,3}=-1$ \\
		$(0,4)$ & $3$ & even & $\hprim^{2,1}=\hprim^{1,2}=5$ & $\hst^{3,2}=\hst^{2,3}=-5$ \\
		$(1,4)$ & $2$ & odd & $\hprim^{2,0}=\hprim^{0,2}=5$, $\hprim^{1,1}=50$ & $\hst^{3,3}=-46$, $\hst^{4,2}=\hst^{2,4}=-5$ \\
		$(0,5)$ & $4$ & odd & $\hprim^{3,1}=\hprim^{1,3}=21$, $\hprim^{2,2}=141$ & $\hst^{3,3}=-138$, $\hst^{4,2}=\hst^{2,4}=-21$ \\
		\bottomrule
	\end{tabular}
\end{center}
\renewcommand {\arraystretch}{1}
\subsection{Stringy Euler numbers of $Y_{m,d}$ for $4\le m+d \le 8$ and $0\le m\le d-1$.} 
For the $16$ members of the family $Y_{m,d}$  listed below, the stringy Euler number \eqref{ESTRINGY} is recorded in the following table:\vspace{0.3cm}
\renewcommand {\arraystretch}{1.4}
\begin{center}
	\begin{tabular}{@{}llllll@{}}
		\toprule
		$\dim$ & \multicolumn{5}{l}{$(m,d)$ and the stringy Euler numbers $\estr(Y_{m,d})$} \\
		\midrule
		$4$ & $(0,4)$: $152$ & $(1,3)$: $18$ & & &\\
		$5$ & $(0,5)$: $-2448$ & $(1,4)$: $-204$ & $(2,3)$: $4$ & &\\
		$6$ & $(0,6)$: $47092$ & $(1,5)$: $4553$ & $(2,4)$: $172$ & &\\
		$7$ & $(0,7)$: $-1016482$ & $(1,6)$: $-105552$ & $(2,5)$: $-4756$ & $(3,4)$: $-28$ &\\
		$8$ & $(0,8)$: $24511888$ & $(1,7)$: $2687520$ & $(2,6)$: $142514$ & $(3,5)$: $2575$ &\\
		\bottomrule
	\end{tabular}
\end{center}
\renewcommand {\arraystretch}{1}\vspace{0.3cm}
For all listed members, $\indst=1$, so the stringy index does not separate them; $\estr$ does.
\section{\texorpdfstring{$K$-equivalence}{K-equivalence}, Hard Lefschetz and smoothability}\label{Comments}
 \noindent $\bullet$ \textbf{$E_{\text{st}}$ as invariant of $K$-equivalence.} The members of the constructed family $(Y_{m,d}\times (\mathbb{P}%
 ^{1})^{n})_{(m,d,n)\text{ }}$ (with $m\geq 0,d\geq 3$ and $n\geq 1$) are
 rational and hence the equidimensional ones are pairwise birationally
 equivalent. This does not conflict with their different stringy $E$%
 -functions. The key point is the hierarchy 
 \begin{equation*}
 	\text{isomorphic}\Longrightarrow K\text{-equivalent} \Longrightarrow  \text{%
 		birationally equivalent,}
 \end{equation*}%
 with $E_{\text{st}}$ constant on the middle relation. Let us repeat the
 definition of $K$-equivalence. Two normal $\mathbb{Q}$-Gorenstein varieties $%
 Y$ and $Y^{\prime }$ are said to be $K$-\textit{equivalent}, written $Y\sim
 _{K}Y^{\prime },$ if there is a smooth variety $W$ with proper birational
 morphisms $\vartheta :W\longrightarrow Y$ and $\vartheta ^{\prime
 }:W\longrightarrow Y^{\prime },$ such that $\vartheta ^{\ast
 }K_{Y}=\vartheta ^{\prime \ast }K_{Y^{\prime }}$ as $\mathbb{Q}$-divisors. A
 birational morphism $f:Y^{\prime }\longrightarrow Y$ with $f^{\ast
 }K_{Y}=K_{Y^{\prime }}$ is \textit{crepant. }If  $f$ is crepant, then $Y\sim
 _{K}Y^{\prime }.$ 
 
 \begin{theorem}[cf.~Batyrev \cite{Bat99a} and Denef--Loeser \cite{DL99}]
 	\label{BATKEQUIV}If $Y,Y^{\prime }$ are normal $\mathbb{Q}$-Gorenstein varieties with
 at worst log-terminal singularities, then    
 \begin{equation*}
 	Y\sim _{K}Y^{\prime }\Longrightarrow E_{\emph{st}}(Y;u,v)=E_{\emph{st}%
 	}(Y^{\prime };u,v).
 \end{equation*}
 \end{theorem}
 
\begin{example}[Strictness of the first implication: $K$-equivalent, not isomorphic]
	Let $Y$ be a projective threefold with one ordinary double point admitting two small resolutions
	$\Ytil_1\to Y\leftarrow\Ytil_2$ related by the Atiyah flop. Both are crepant, hence
	\[\Ytil_1\sim_K\Ytil_2\sim_KY \Longrightarrow \E(\Ytil_1;u,v)=E(\Ytil_2;u,v)=\Est(Y;u,v),\]
although $\Ytil_1\not\cong\Ytil_2$ in general. This is the
	mechanism behind the whole crepant-resolution philosophy.
\end{example}

\begin{example}[Strictness of the second implication: birational, not $K$-equivalent]\label{ex:P2}
	$\PP^2$ and $\PP^1\times\PP^1$ are smooth, rational and birationally equivalent, yet
	\[
E(\PP^2;u,v)=\Est(\PP^2;u,v)=1+uv+(uv)^2\ \neq\ 1+2uv+(uv)^2=\Est(\PP^1\times\PP^1;u,v)=E(\PP^1\times\PP^1;u,v).
	\]
So $\Est$ already fails to be birationally invariant on
	\emph{smooth} varieties, before any singularity is in play.
\end{example}

\begin{remark}[Why the Calabi--Yau intuition misleads] One has $\Est(Y;u,v)=E(Y;u,v)$ for smooth $Y$'s, so $\Est$ is a birational invariant only in situations
	where birational equivalence forces $K$-equivalence. This occurs when $K_Y\sim0$; Theorem \ref{BATKEQUIV} then specialises to Batyrev's theorem that birational Calabi--Yau manifolds have
	equal Betti numbers, cf.\ \cite{Bat99b}. More generally, it reflects the fact that  minimal models of a fixed klt pair are crepant birational; in the case
	of the pair $(Y, 0)$ this gives the corresponding $K$-equivalence statement, under the usual
	$\mathbb Q$-factoriality and MMP hypotheses. (See, for example, the discussion surrounding \cite[Thm.~3.52]{KM98}). The family $Y_{m,d}$ lies at
	the opposite extreme.  Since $Y_{m,d}$
	is rational, every smooth projective model in its birational class is
	uniruled. Hence its canonical divisor is not pseudo-effective
	(cf.~\cite[Thm. 0.2 and Cor. 0.3]{BDPP}), and the $K$-MMP terminates with a Mori fibre space
	rather than a minimal model (cf.~\cite[Cor.~1.3.3]{BCHM}). Thus there
	is no minimal-model mechanism forcing the members of this birational
	class to be $K$-equivalent. That a single birational class --- here the class of $\PP^{m+d}$ --- contains many
	$K$-classes with markedly different stringy invariants is therefore consistent with the standard MMP picture.
\end{remark}

 \noindent $\bullet$ \textbf{The unifying role of the family.} 
The significance of the family $(Y_{m,d}\times (\mathbb{P}%
^{1})^{n})_{(m,d,n)\text{ }}$ goes beyond providing a common notation for several examples. Its construction is motivated by the counterexamples of Huang and Satriano \cite{HS26} (with varying values of $n$), which serve as a prototype for the geometric mechanism developed here: a singular Fano variety with a simple linear singular locus, whose blow-up produces an exceptional divisor carrying the primitive cohomology responsible for the failure of non-negativity. Varying the parameters \(m\) and \(d\) naturally extends this prototype, allowing both the dimension of the singular locus and the geometry of the exceptional divisor to vary while preserving the same basic resolution pattern.  The geometry of the singular locus and its resolution, the stringy \(E\)-function, and the mechanism producing negative stringy Hodge numbers are then governed uniformly by the parameters $(m,d,n)$. Previously known counterexamples, genuinely new members, exceptional transition cases, and infinite regimes of persistent negativity consequently appear as different regions of a single parameter lattice.\medskip

 \noindent $\bullet$ \textbf{What separates the members?}  Besides their dimension and transverse singularity type, the members of the family can be distinguished by three further kinds of geometric data: their \(K\)-equivalence classes, their polarised Fano invariants, and the Hodge–Deligne data of the distinguished log resolutions constructed in Theorem \ref{thm:uniform-geometry}. 
 \smallskip\newline
 \noindent (i) \textit{Stringy $E$-functions as obstructions to $K$-equivalence.} By Theorem \ref{BATKEQUIV}, any two members of the family having distinct stringy $E$-functions are not $K$-equivalent. Thus the variation of $E_{\text{st}}$ separates many $K$-equivalence classes among birational members of the family. (The stringy $E$-function is invariant under
 $K$-equivalence, but not under ordinary birational equivalence;
 see~ \cite{Bat99a} and \cite{DL99}.)  For the subfamily $(Y_{m,d})_{(m,d)}$ of fixed dimension $m+d$, pairwise non-$K$-equivalence can also be seen from the anticanonical degrees (volumes) \eqref{antican1}. If we fix the dimension $\tilde{d}=m+d$ and define $\mathfrak{f}(m):=(-K_{Y_{m,d}})^{m+d}= d(m+2)^{m+d}$ for the consecutive admissible values $m=0,1,\dots,\big\lfloor\frac{\tilde{d}-1}{2}\big\rfloor$,  we see that
 \[\frac{\mathfrak{f}(m+1)}{\mathfrak{f}(m)}
 =\frac{\tilde{d}-m-1}{\tilde{d}-m}\big(\frac{m+3}{m+2}\big)^{\tilde{d}}
 \ge\frac{d-1}{d}\big(1+\frac{1}{m+2}\big)^{m+3}>\frac23\cdot2>1,\] using $\tilde{d}\ge m+3$,
 $(1+\frac1\xi)^{\xi+1}>\text{e}>2$ for all $\xi\in\mathbb R_{>0}$, and $d\ge3$;  hence $\mathfrak{f}$ is strictly increasing, and in
 particular injective, so the members of fixed dimension $\tilde{d}$ have pairwise distinct
 anticanonical degrees.  On the other hand, the anticanonical degree is a $K$-equivalence
 invariant: if $Y_{m,d}\sim_K Y_{m',d'}$ via $\vartheta\colon W\to Y_{m,d}$ and $\vartheta'\colon W\to Y_{m',d'}$ with
 $\vartheta^*K_{Y_{m,d}}=\vartheta'^*K_{Y_{m',d'}}$, then \[(-K_{Y_{m,d}})^{\tilde{d}}=(-\vartheta^*K_{Y_{m,d}})^{\tilde{d}}
 =(-\vartheta'^*K_{Y_{m',d'}})^{\tilde{d}}=(-K_{Y_{m',d'}})^{\tilde{d}}\] by the projection formula, with both morphisms being
 birational.  Distinct volumes therefore force non-$K$-equivalence.
 \smallskip \newline
\noindent (ii)  \textit{The polarised Fano class.} For $Y_{m,d}$ we have
\begin{equation*}
-K_{Y_{m,d}}=(m+2)\mathfrak{H},\ \ \text{deg}(Y_{m,d})=d=\mathfrak{H}^{m+d},\ \text{ }(-K_{Y_{m,d}})^{m+d}=d(m+2)^{m+d},
\end{equation*}%
and the numerical quantity corresponding to coindex is $d-1$. These are standard coarse invariants used to place Fano varieties in classification tables. In particular, the members with \(d=3\) numerically occupy Fujita's coindex-\(2\) (del Pezzo) range; a smooth del Pezzo \(n\)-fold of degree \(3\) is a cubic hypersurface in \(\mathbb P^{n+1}\) \cite[Ch.~I, Thm.~8.11(3), p. 72]{Fuj90}. The members with \(d=4\)  numerically occupy coindex-\(3\) (Mukai) range: they occur in the classification tables for genus \(g=3\). In the very ample (first-species) case the corresponding smooth varieties are quartic hypersurfaces in \(\mathbb P^{n+1}\); the genus-\(3\) classification also contains the double-cover case over a smooth quadric \cite[Prop.~1, p.~3001]{Muk89}. The family $(Y_{m,d})_{(m,d)}$ thus provides singular representatives inside classical Fano ranges, one for each $m$.
\smallskip \newline
\noindent  (iii)  \textit{The Hodge-Deligne data of the resolution.} The discrepancy of the log resolution of Theorem \ref{thm:uniform-geometry}, $f\colon \widetilde Y_{m,d}\longrightarrow Y_{m,d}$ is uniformly one. The same holds after taking products \[f\times\mathrm{id}:\widetilde Y_{m,d}\times(\PP^1)^n\longrightarrow
Y_{m,d}\times(\PP^1)^n.\] What changes dramatically is the geometry of  $D_{m,d}$.  The exceptional geometry enters directly into the one-divisor formula and, together with the Hodge-Deligne polynomial of the resolution, determines the stringy $E$-function. In the present family these data vary substantially even though all members are rational. \medskip

\noindent $\bullet$ \textbf{Stringy Hard Lefschetz?}  It is natural to ask if there is an analogue of  \eqref{HLIREL} at the level of stringy Hodge numbers.
\begin{question}[Hard Lefschetz property for stringy Hodge numbers]\label{QHARD} If $V$ is a
	projective complex variety with Gorenstein canonical singularities and
	polynomial $E_{\emph{st}}$-function,  are the inequalities 
	\begin{equation*}
		\fbox{$%
			\begin{array}{ccc}
				\  & h_{\emph{st}}^{p,q}(V)\leq h_{\emph{st}}^{p+1,q+1}(V) & \ 
			\end{array}%
			$}\text{ }
	\end{equation*}%
valid for its stringy Hodge numbers for all integers $p\geq 0$ and $q\geq 0$ with $%
	p+q\leq $ \emph{dim}$(V)-2?$
\end{question}
\noindent An affirmative answer to the above question was given by Schepers \cite[\S 3]{S09} in the
case where either dim$(V)=3,$ or dim$(V)\geq 4,$ $V$ has at most \textit{%
	isolated }singularities, and $V$ admits a log desingularisation whose
exceptional components all have discrepancy coefficients greater than  $
\left \lfloor \tfrac{1}{2}(\text{dim}(V)-4)\right \rfloor .$  On the other hand,
Schepers  \cite[Ex. 4.3]{S09} constructed a $6$-dimensional projective variety \(Y\) with a single isolated canonical hypersurface singularity whose stringy Hodge numbers \textit{do not have the Hard Lefschetz property}. This counterexample is derived from a partial desingularisation of the hypersurface 
\begin{equation*}
	\left \{ \lbrack x_{0}:x_{1}:\cdots :x_{7}]\in \PP^{7}\left \vert
	x_{1}^{5}x_{0}^{3}+x_{2}^{5}x_{0}^{3}+\sum_{j=3}^{7}x_{j}^{8}=0\right.
	\right \} 
\end{equation*} leaving it with only one \textit{isolated} singularity. The resolution of this singular point is carried out by a sequence of blow-ups. The resulting exceptional divisor has normal crossings with discrepancy coefficients $0$ or $1$ (whereas $\tfrac{6-4}{2}=1$). Another negative answer to Question \ref{QHARD} was given by Musta\c{t}\v{a} and Payne \cite[Ex. 1.1]{MuP05} by constructing a $6$-dimensional projective toric variety with $h^{3,3}_{\text{st}}<h^{2,2}_{\text{st}}.$ As a by-product of our computation mechanism (cf. proof of Theorem \ref{thm:classification}, subcase 2A) Theorem \ref{SPMEMBER} shows that there are \textit{seventeen} more counterexamples $Y_{2,3}\times (\PP^{1})^{n},\ 1\le n \le 17,$ \textit{of a different nature}.  These are non-toric, have non-isolated singularities, and are terminal. For the \textit{fourteen} values  $4\le n\le 17$,  the stringy Hodge numbers are moreover non-negative. Thus failure of non-negativity and failure of the Hard Lefschetz property are independent phenomena, and a single one-parameter subfamily separates them: along
$n\mapsto Y_{2,3}\times(\PP^1)^n$, non-negativity of the stringy Hodge numbers is restored
at $n=4$, whereas Hard Lefschetz property is restored only at $n=18$.
\begin{theorem}[The subfamily $Y_{2,3}\times (\PP%
	^{1})^{n}$ and the Hard Lefschetz threshold]\label{SPMEMBER}Since 
\begin{equation}
	E_{\emph{st}}(Y_{2,3}\times (\PP%
	^{1})^{n};u,v)=(1+uv)^{n+1}(1+(uv)^{2})^{2}-8(uv)^{3}(1+uv)^{n-1} \label{Y23(1)}
\end{equation}%
we have $h_{\emph{st}}^{p,q}(Y_{2,3}\times (\PP^{1})^{n})=0$ for $p\neq q$ and\emph{\footnote{As usual, for integers $a,b$ with $a\geq 0,$ we define $\binom{a}{b}:=0$ for $b<0$ or $b>a.$}}
\begin{equation}
	h_{\emph{st}}^{r,r}(Y_{2,3}\times (\PP^{1})^{n})=\binom{n+1}{r}+2\binom{n+1}{%
		r-2}+\binom{n+1}{r-4}-8\binom{n-1}{r-3}.  \label{Y23(2)}
\end{equation}
\emph{(i)} If $n\geq 4,$ then for any integer $r,$ 
\begin{equation}
	h_{\emph{st}}^{r,r}(Y_{2,3}\times (\PP^{1})^{n})\geq 0.  \label{Y23(3)}
\end{equation}
\emph{(ii)} If $n\geq 18,$ then 
\begin{equation}
	h_{\emph{st}}^{0,0}(Y_{2,3}\times (\PP^{1})^{n})<h_{\emph{st}%
	}^{1,1}(Y_{2,3}\times (\PP^{1})^{n})<\cdots <h_{\emph{st}}^{r,r}(Y_{2,3}\times (\PP^{1})^{n})\  \  \
	\left( 0\leq r\leq \left \lfloor \frac{n+5}{2}\right \rfloor \right) .\label{2318}
\end{equation}
Moreover, the lower bound $18$ is sharp: for every integer $n$ with $1\leq
n\leq 17,$ the monotonicity in \emph{\eqref{2318}} fails and the Hard Lefschetz property is not valid for the stringy Hodge numbers.
\end{theorem}
\noindent Before giving the proof, we shall first establish three computational lemmas.
\begin{lemma}For $n=18,$ \label{lemma18}
\begin{equation*}
	h_{\emph{st}}^{0,0}(Y_{2,3}\times (\PP^{1})^{18})<h_{\emph{st}%
	}^{1,1}(Y_{2,3}\times (\PP^{1})^{18})<h_{\emph{st}}^{2,2}(Y_{2,3}\times (\PP%
	^{1})^{18})<\cdots <h_{\emph{st}}^{11,11}(Y_{2,3}\times (\PP^{1})^{18}).\ 
\end{equation*}
\end{lemma}
\begin{proof}
	Since $n+5=23$, the midpoint lies between indices $11$ and $12$. Direct substitution into \eqref{Y23(2)} gives the following exact values:
	{\small
	\begin{center}
		\begin{tabular}{c|rrrrrrrrrrrr}
			\toprule
			$r$ & 0 & 1 & 2 & 3 & 4 & 5 & 6 & 7 & 8 & 9 & 10 & 11\\
			\midrule
			{\small $h_{\text{st}}^{r,r}(Y_{2,3}\times (\PP^{1})^{18})$}
			&1&19&173&999&4083&12497&29615&55573&84218&105774&115090 & 116246\\
			\bottomrule
		\end{tabular}
	\end{center}}
	\noindent Every entry is strictly larger than the preceding one. Hence the claimed strict increase holds. By symmetry,
	{\small $h_{\text{st}}^{12,12}(Y_{2,3}\times (\PP^{1})^{18})=h_{\text{st}}^{11,11}(Y_{2,3}\times (\PP^{1})^{18})$} and the remaining coefficients decrease in reverse order.
\end{proof}
\begin{lemma} \label{ESTY23}For every $n\geq 1,$%
\begin{equation}
	E_{\emph{st}}(Y_{2,3}\times (\PP^{1})^{n+1};u,v)=\left( 1+uv\right) E_{\emph{%
			st}}(Y_{2,3}\times (\PP^{1})^{n};u,v).  \label{ESTINDUCTION}
\end{equation}%
Equivalently, at the level of coefficients, 
\begin{equation}
	h_{\emph{st}}^{r,r}(Y_{2,3}\times (\PP^{1})^{n+1})=h_{\emph{st}%
	}^{r,r}(Y_{2,3}\times (\PP^{1})^{n})+h_{\emph{st}}^{r-1,r-1}(Y_{2,3}\times (%
	\PP^{1})^{n}).  \label{HRRIND}
\end{equation}
\end{lemma}

\begin{proof}Formula \eqref{ESTINDUCTION} follows immediately, and comparison of the coefficient of $%
(uv)^{r}$ on both sides yields \eqref{HRRIND}. \end{proof}
\begin{lemma}[Multiplication by $1+t$ preserves symmetric unimodality]\label{lem:unimodal}
	Let 
	$
	\Theta(t)=\sum_{j=0}^{\kappa }a_jt^j
	$
	be a polynomial with non-negative integer coefficients. Assume that its coefficients are symmetric,
	$
	a_j=a_{\kappa -j},
	$
	and are non-decreasing up to the midpoint,
	$
	a_0\leq a_1\leq\cdots\leq a_{\lfloor \kappa /2\rfloor}.
	$
	Write
	\[
	(1+t)\Theta(t)=\sum_{j=0}^{\kappa +1}b_jt^j.
	\]
	Then the coefficients $b_j$ are again symmetric and non-decreasing up to the midpoint. If, moreover, $\Theta \neq 0$ and
	$
	a_0<a_1<\cdots<a_{\lfloor \kappa /2\rfloor},
	$
	then
	$
	b_0<b_1<\cdots<b_{\lfloor(\kappa +1)/2\rfloor}.
	$
\end{lemma}

\begin{proof}
	Put $a_{-1}=a_{\kappa +1}=0$. Then
	\begin{equation}\label{eq:b}
		b_j=a_j+a_{j-1}.
	\end{equation}
	Symmetry follows immediately:
$
	b_{\kappa +1-j}=a_{\kappa +1-j}+a_{\kappa -j}=a_{j-1}+a_j=b_j.
$	For monotonicity, we subtract consecutive coefficients. From \eqref{eq:b},
$
		b_{j+1}-b_j=a_{j+1}-a_{j-1}.
$
	If $\kappa =2\lambda$ is even, then for $0\leq j\leq \lambda-1$ both indices $j-1$ and $j+1$ lie on the non-decreasing side of the coefficient sequence (with $a_{-1}=0$ when $j=0$). Hence
	$a_{j+1}\geq a_{j-1}$, and therefore $b_{j+1}\geq b_j$. If $\kappa =2\lambda+1$ is odd, the same argument works for $0\leq j\leq \lambda-1$. For the last required comparison, namely $j=\lambda$, symmetry gives
	$
	a_{\lambda+1}=a_\lambda,
	$
	and therefore
	\[
	b_{\lambda+1}-b_\lambda=a_{\lambda+1}-a_{\lambda-1}=a_\lambda-a_{\lambda-1}\geq0.
	\]
	This proves weak monotonicity. If the inequalities for the $a_j$ are strict up to the midpoint, then every displayed difference is strictly positive, giving the final assertion.
\end{proof}
\medskip
\noindent \textbf{Proof of Theorem \ref{SPMEMBER}.}  $Y_{2,3}$ is a general cubic fivefold in $\PP^{6}$ whose singular locus is a plane $\Lambda \cong\mathbb P^{2}$, along which it has
multiplicity $2$. Here $%
Z_{2,3}\subset \PP^{3}$ consists of $8$ reduced points and $Z_{2,3}^{\prime
}=\varnothing .$ The master formula  \eqref{masterfor} for the $E_{\text{st}}$-polynomial of
the $\left( 5+n\right) $-dimensional product $Y_{2,3}\times (\PP^{1})^{n}$
gives \eqref{Y23(1)}. This is a polynomial in $uv$ only.  \eqref{Y23(2)} follows after expanding $%
(1+uv)^{n+1}$ and $(1+uv)^{n-1}$ on the right-hand side of \eqref{Y23(1)}. 
\smallskip \newline
\noindent (i) \eqref{Y23(3)} is an immediate consequence of Theorem \ref{thm:classification}.
\smallskip \newline
\noindent (ii) By Lemma \ref{lemma18} the polynomial $E_{\text{st}}(Y_{2,3}\times (\PP%
^{1})^{18};u,v)$  has symmetric coefficients which are strictly increasing
up to the midpoint. For every $n\geq 18,$ 
\begin{equation*}
	E_{\text{st}}(Y_{2,3}\times (\PP^{1})^{n};u,v)=(1+uv)^{n-18}E_{\text{st}%
	}(Y_{2,3}\times (\PP^{1})^{18};u,v).
\end{equation*}%
Starting with $E_{\text{st}}(Y_{2,3}\times (\PP^{1})^{18};u,v)$ and
multiplying successively by $1+uv,$ formula  \eqref{ESTINDUCTION} of  Lemma \ref{ESTY23} and Lemma \ref{lem:unimodal} show inductively that symmetry
and strict increase up to the midpoint are preserved at each step. Therefore
\eqref{2318} holds for every $n\geq 18.$ It remains to show that the bound $n\geq18$ cannot be lowered. For each
$n\in\{1,2,...,17\}$, it suffices to exhibit one index $r\in\{0,1,...,\left\lfloor\frac{n+5}{2}\right\rfloor -1\}$
for which
\[
h_{\text{st}%
}^{r,r}(Y_{2,3}\times (\PP^{1})^{n})>h_{\text{st}%
}^{r+1,r+1}(Y_{2,3}\times (\PP^{1})^{n}).
\]
The following table gives such a violation for every $n$ in the indicated range. All entries are obtained by direct substitution into \eqref{Y23(2)}.
\renewcommand {\arraystretch}{1.3} {\footnotesize 
	\begin{equation*}
		\begin{tabular}{|c|c|c|c|}
			\hline
			$n$ & $r$ & $h_{\text{st}}^{r,r}(Y_{2,3}\times (\PP^{1})^{n})$ & $h_{\text{%
					st }}^{r+1,r+1}(Y_{2,3}\times (\PP^{1})^{n})$ \\ \hline \hline
			$1$ & $2$ & $3$ & $-4$ \\ \hline
			$2$ & $2$ & $5$ & $-1$ \\ \hline
			$3$ & $2$ & $8$ & $4$ \\ \hline
			$4$ & $3$ & $12$ & $2$ \\ \hline
			$5$ & $3$ & $24$ & $14$ \\ \hline
			$6$ & $3$ & $41$ & $38$ \\ \hline
			$7$ & $4$ & $79$ & $56$ \\ \hline
			$8$ & $4$ & $143$ & $135$ \\ \hline
			$9$ & $5$ & $278$ & $227$ \\ \hline
		\end{tabular}%
		\  \ 
		\begin{array}{ccc}
			&  & 
		\end{array}%
		\begin{tabular}{|c|c|c|c|}
			\hline
			$n$ & $r$ & $h_{\text{st}}^{r,r}(Y_{2,3}\times (\PP^{1})^{n})$ & $h_{\text{%
					st }}^{r+1,r+1}(Y_{2,3}\times (\PP^{1})^{n})$ \\ \hline \hline
			&  &  &  \\ \hline
			$10$ & $5$ & $515$ & $505$ \\ \hline
			$11$ & $6$ & $1020$ & $916$ \\ \hline
			$12$ & $7$ & $1936$ & $1738$ \\ \hline
			$13$ & $7$ & $3840$ & $3674$ \\ \hline
			$14$ & $8$ & $7514$ & $7150$ \\ \hline
			$15$ & $8$ & $14690$ & $14664$ \\ \hline
			$16$ & $9$ & $29354$ & $28964$ \\ \hline
			$17$ & $10$ & $58318$ & $57928$ \\ \hline
		\end{tabular}%
	\end{equation*}%
} \renewcommand {\arraystretch}{1}
\noindent \hspace{-0.2cm}Thus, for every $n\in\{1,2,...,17\}$, the coefficient sequence decreases somewhere before reaching the midpoint. Consequently, the monotonicity in \emph{\eqref{2318}} fails for every one of these values of $n$. For $n=1$ and $n=2$ the right-hand entry is $-4$ and $-1$, respectively, so the inequality fails only because non-negativity already fails. For $n=3$ the failure $8>4$ is between two positive stringy Hodge numbers, so that particular violation is not a by-product of negativity $-$  although $h_{\text{st}%
}^{4,4}(Y_{2,3}\times (\PP^{1})^{3})=-2<0$. The first value of $n$ for which Hard Lefschetz fails while \textit{all} stringy Hodge numbers are non-negative is $n=4$. On the other hand, at the final exceptional value $n=17$, one has
\[
h_{\text{st}%
}^{10,10}(Y_{2,3}\times (\PP^{1})^{17})=58318 > 57928=h_{\text{st}%
}^{11,11}(Y_{2,3}\times (\PP^{1})^{17}).
\]
So the failure occurs immediately before the midpoint. At $n=18$ this last decrease disappears:
\[
h_{\text{st}%
}^{10,10}(Y_{2,3}\times (\PP^{1})^{18})=115090<116246=h_{\text{st}%
}^{11,11}(Y_{2,3}\times (\PP^{1})^{18}).
\]
Hence $n=18$ is the exact threshold.\hfill $\Box$
\medskip \newline 
\noindent A simple way to produce an \textit{infinite} family of counterexamples to Question \ref{QHARD} is as follows: We fix $n=2$, take the boundary case $m=d-1$, and let $d\geq 3$ vary.

\begin{theorem}For every $d\ge 3$, $Y_{d-1,d}\times (\mathbb{P}^{1})^{2}$ does not have the Hard
Lefschetz property.\end{theorem}
\begin{proof}
By formula \eqref{ESTD-1D}, 
\begin{equation*}
	E_{\text{st}}(Y_{d-1,d}\times (\mathbb{P}^{1})^{2};u,v)=(1+uv)[(1+uv+\cdots
	+(uv)^{d})^{2}-(d-1)^{d}(uv)^{d}].
\end{equation*}%
Since $h_{\text{st}}^{d-1,d-1}(Y_{d-1,d}\times (\mathbb{P}^{1})^{2})=2d-1$
and $h_{\text{st}}^{d,d}(Y_{d-1,d}\times (\mathbb{P}%
^{1})^{2})=2d+1-(d-1)^{d},$ we obtain 
\begin{equation}
	h_{\text{st}}^{d-1,d-1}(Y_{d-1,d}\times (\mathbb{P}^{1})^{2})-h_{\text{st}%
	}^{d,d}(Y_{d-1,d}\times (\mathbb{P}^{1})^{2})=(d-1)^{d}-2>0 \label{HLVIOL}
\end{equation}%
with $2(d-1)=2d-2<\dim (Y_{d-1,d}\times (\mathbb{P}%
^{1})^{2})-2=[(d-1)+d+2]-2=(2d+1)-2.$ \eqref{HLVIOL} violates a Hard Lefschetz inequality for stringy Hodge numbers. \end{proof}
\noindent $\bullet$ \textbf{Interpretation in terms of flat degenerations.} $Y_{m,d}=V(F)\subset \mathbb{P}^{N},$ $N:=m+d+1,$ is a hypersurface of
degree $d$ and $\dim (Y_{m,d})=m+d,$ $F\in H^{0}(\mathbb{P}^{N},\mathcal{O}_{%
	\mathbb{P}^{N}}(d)).$ By Bertini's theorem, the smooth members of the
base-point-free complete linear system $\left \vert \mathcal{O}_{\mathbb{P}%
	^{N}}(d)\right \vert $ form a non-empty Zariski-open subset. (Cf. \cite[Ch. III, Cor. 10.9]{Har77}.)  Hence we may choose
a form $G\in H^{0}(\mathbb{P}^{N},\mathcal{O}_{\mathbb{P}^{N}}(d)),$ such
that $Y_{\text{gen}}:=V(G)$ is smooth, and $G$ is not proportional to $F.$  Then $\mathbb P(\langle F,G\rangle)\simeq \mathbb P^1 $ is a projective line in
$\left \vert \mathcal{O}_{\mathbb{P}^{N}}(d)\right \vert $. (Cf. \cite[Ch. II, \S 7]{Har77}.) We consider the pencil generated by \(F\) and \(G\). Thus, with homogeneous coordinates
$ [s:t]\in\mathbb P^1, $ we consider the family
\[ \mathcal Y := \left\{ (x,[s:t])\in\mathbb P^N\times\mathbb P^1 \; ;\; sF(x)+tG(x)=0 \right\}. \]
Let $\mathcal Y\longrightarrow\mathbb P^1 $ be the projection onto the second factor. The fibre over \([s:t]\) is
\[ Y_{[s:t]} = V(sF+tG)\subset\mathbb P^N. \]
In particular, $ Y_{[1:0]}=V(F)=Y_{m,d}, $ whereas $ Y_{[0:1]}=V(G)=Y_{\mathrm{gen}} $ is smooth. Thus the family contains both our singular hypersurface \(Y_{m,d}\) and a smooth degree-\(d\) hypersurface. The variety $\mathcal Y\subset\mathbb P^N\times\mathbb P^1$ is a divisor of bidegree $(d,1)$. For general $G$, it is irreducible and dominates $\mathbb P^1$. A non-constant morphism from an integral scheme to a smooth curve is flat. Moreover, all fibres are degree-$d$ hypersurfaces and hence have the same Hilbert polynomial. Finally, Bertini's theorem implies that the general fibre is smooth:

\begin{theorem}\label{DEFTHM}
	The hypersurface $Y_{m,d}\subset\PP^{N}$ is smoothable; more precisely, it is a flat degeneration of a smooth Fano hypersurface $Y_{\mathrm{gen}}\subset\PP^{N}$ of
	degree $d$. The smooth fibre has index $N+1-d=m+2$ and the same anticanonical degree $d(m+2)^{m+d}$. Likewise,
	$Y_{m,d}\times(\PP^1)^n$ is a flat degeneration of $Y_{\mathrm{gen}}\times(\PP^1)^n$.
\end{theorem}

Thus these are smoothable Gorenstein terminal Fano varieties which, after fixing the relevant projective embeddings, lie in the same irreducible component of the Hilbert scheme as their smooth counterparts. This observation is useful for two reasons. First, for fixed values of the parameters, it places the singular member in the same well-understood bounded family as the corresponding smooth hypersurfaces. Secondly, it shows that any invariant remaining constant along the above deformation cannot by itself account for a phenomenon occurring only on the singular special fibre. Such a phenomenon must instead involve data which are sensitive to the singularities.

\begin{remark}[Interpretation of deformation-invariant versus singularity-sensitive invariants]Suppose that an invariant, defined for both the special and general fibres, remains constant along the flat family considered above. Then it takes the same value on \(Y_{m,d}\) as on a nearby smooth hypersurface. Consequently, any phenomenon that occurs on the singular special fibre but not on the smooth general fibre cannot be detected by such deformation-invariant data alone. In the applications motivating Theorem \ref{DEFTHM}, the exceptional behaviour must therefore arise from data sensitive to the singularities of \(Y_{m,d}\). If, as established earlier, these contributions are encoded by the divisor \(D_{m,d}\) on the relevant resolution, then the latter provides the geometric source of the phenomenon.
\end{remark}

\appendix
\section{Hodge-theoretic computational tools}\label{app:hodge}

\subsection{Hodge and Lefschetz decompositions\label{SUNSA1}}
Let $V$ be a smooth projective complex variety of complex dimension $n\geq 1$
embedded in $\PP^{\nu }$ and let $\mathfrak{h}:=c_{1}(\mathcal{O}_{V}(1))\in
H^{2}(V,\mathbb{Z})$ be the restriction of the hyperplane class of  $\PP%
^{\nu }$ to $V.$ Since $V$ is compact K\"{a}hler,  $H^{j}(V,\mathbb{C})$ has the \textit{Hodge decomposition}%
\begin{equation*}
        H^{j}(V,\mathbb{C})\cong \bigoplus \limits_{p+q=j}H^{q}(V,\Omega _{V}^{p}),%
        \text{ where }h^{p,q}(V):=\dim _{\mathbb{C}}\left( H^{q}(V,\Omega
        _{V}^{p}\right) )
\end{equation*}%
are the \textit{Hodge numbers} of $V$ (with $h^{p,q}(V)=h^{q,p}(V)$  and $h^{p,q}(V)=h^{n-p,n-q}(V)$). If $\mathcal{L}$ is the Lefschetz
operator, i.e. the cup product with $\mathfrak{h}$,
$
        \mathcal{L}:H^{j}(V,\mathbb{C})\longrightarrow H^{j+2}(V,\mathbb{C}),\text{ }%
        \xi \longmapsto \mathfrak{h}\smile \xi ,
$
then for all $r\in \mathbb{Z}_{\geq 0}\mathfrak{\ }$%
\begin{equation*}
        \mathcal{L}^{r}:H^{j}(V,\mathbb{C})\longrightarrow H^{j+2r}(V,\mathbb{C}),%
        \text{ }\xi \longmapsto \mathfrak{h}^{r}\smile \xi ,
\end{equation*}%
and the $j$-th \textit{primitive cohomology group }is defined as follows: 
\begin{equation*}
        H_{\text{prim}}^{j}(V,\mathbb{C}):=\text{Ker}\left( \mathcal{L}%
        ^{n-j+1}:H^{j}(V,\mathbb{C})\longrightarrow H^{2n-j+2}(V,\mathbb{C})\right)
        ,\  \  \forall j\in \{0,...,n\}.
\end{equation*}%
The \textit{Lefschetz decomposition} gives
\begin{equation*}
        H^{j}(V,\mathbb{C})=\bigoplus \limits_{r\geq \max \{0,j-n\}}%
        \mathcal{L}^{r}(H_{\text{prim}}^{j-2r}(V,\mathbb{C})).
\end{equation*}%
Since the Lefschetz and Hodge decompositions are compatible,  we have
\begin{equation*}
        H_{\text{prim}}^{j}(V,\mathbb{C})=\bigoplus \limits_{p+q=j}H_{\text{prim}%
        }^{p,q}(V,\mathbb{C}),\text{ where }H_{\text{prim}}^{p,q}(V,\mathbb{C}%
        ):=H^{q}(V,\Omega _{V}^{p})\cap \text{Ker}(\mathcal{L}^{n-j+1}),
\end{equation*}%
and $h_{\text{prim}}^{p,q}(V):=\dim _{\mathbb{C}}(H_{\text{prim}}^{p,q}(V,%
\mathbb{C}))$ are the so-called \textit{primitive Hodge numbers} of $V.$ $%
\mathfrak{h}$\textit{\ }has Hodge type $(1,1).$ Lefschetz decomposition
respects Hodge type. Thus, $\mathcal{L}:H^{q}(V,\Omega
_{V}^{p})\longrightarrow H^{q+1}(V,\Omega _{V}^{p+1}),$ and in the range
where primitive pieces are defined,%
\begin{equation*}
        H^{q}(V,\Omega _{V}^{p})\cong \bigoplus \limits_{r\geq \max \{0,p+q-n\}}%
        \mathcal{L}^{r}(H_{\text{prim}}^{p-r,q-r}(V,\mathbb{C})).
\end{equation*}%
Therefore
\begin{equation*}
        h^{p,q}(V)=\sum_{r= \max \{0,p+q-n\}}^{\text{min}{{\{p,q\}}}}h_{\text{prim}}^{p-r,q-r}(V)\text{.}
\end{equation*}%
Moreover, by Hodge symmetry, $h_{\text{prim}}^{p,q}(V)=h_{\text{prim}}^{q,p}(V).$
\begin{theorem}[Hard Lefschetz Theorem ({\cite[Thm. 1.30]{PSt08}})]\label{VERYHARD}
$ \mathcal{L}^{n-k}$ induces an isomorphism \begin{equation*}   \mathcal{L}^{n-k}:H^{k}(V,\mathbb{C})\overset{\cong }{\longrightarrow }%
	H^{2n-k}(V,\mathbb{C})  ,\  \  \forall k\in \{0,...,n\}.  \end{equation*} 
\end{theorem}
\begin{corollary}[Hard Lefschetz Inequalities]
	For all integers $p\geq 0$, $q\geq 0,$ with $p+q\leq n-1$ we have
	\begin{equation}
		\fbox{$%
			\begin{array}{ccc}
				\  & h^{p,q}(V)\leq h^{p+1,q+1}(V). & \  \label{HLIREL}
			\end{array}%
			$}
	\end{equation}%
	$\emph{(}$If $p+q=n-1,$ this  holds as equality.$\emph{)}$
\end{corollary}
\begin{proof}We assume that $p+q=k\leq n-1.$ We first show that $\mathcal{L}:H^{k}(V,%
	\mathbb{C})\longrightarrow H^{k+2}(V,\mathbb{C})$ is injective. Indeed, if $%
	\xi \in $ Ker$(\mathcal{L}),$ then, since $k\leq n-1,$ by Theorem \ref{VERYHARD} $%
	\mathcal{L}^{n-k}$ induces an isomorphism. Hence 
	\begin{equation*}
		\mathcal{L}^{n-k}(\xi )=\mathcal{L}^{n-k-1}(\mathcal{L}(\xi ))=\mathcal{L}%
		^{n-k-1}(0)=0\Longrightarrow \xi \in \text{Ker}(\mathcal{L}^{n-k})=\{0\}%
		\Longrightarrow \xi =0
	\end{equation*}%
	and $\mathcal{L}$ is injective in every degree $k\leq n-1.$ Since $\mathcal{L%
	}$ shifts Hodge type by $(1,1),$ its restriction%
	\begin{equation*}
		\mathcal{L}:H^{q}(V,\Omega ^{p})\longrightarrow H^{q+1}(V,\Omega ^{p+1})
	\end{equation*}%
	is also injective, and the inequalities \eqref{HLIREL} follow. If $\ p+q=n-1,$ then Hard
	Lefschetz, applied in degree $n-1,$ gives directly 
$
		\mathcal{L}:H^{n-1}(V,\mathbb{C})\overset{\cong }{\longrightarrow }H^{n+1}(V,%
		\mathbb{C}).
$
	Since $\mathcal{L}$ has Hodge type $(1,1),$ this isomorphism respects the
	Hodge decompositions and therefore induces%
	\begin{equation*}
		\mathcal{L}:H^{q}(V,\Omega ^{p})\overset{\cong }{\longrightarrow }%
		H^{q+1}(V,\Omega ^{p+1})\text{ \  \ }(p+q=n-1).
	\end{equation*}%
	Consequently, $h^{p,q}(V)=h^{p+1,q+1}(V).$\end{proof}
\subsection{Primitive cohomology and $E$-polynomial of smooth complete intersections\label{SUNSA2}}Now let $V=V_{(d_{1},...,d_{c})}\subset \PP^{\nu }$ be a smooth complete intersection of dimension $n=\nu - c$ with multidegree $(d_{1},...,d_{c})$. If $n\geq 1$, then
the Lefschetz hyperplane theorem, applied successively to
the hypersurfaces cutting out $V$, implies that $H^{j}(\PP^{\nu },\mathbb{Q}%
)\cong H^{j}(V,\mathbb{Q})$ for $j<n.$ By Poincar\'{e} duality the
corresponding statement above the middle degree follows as well. Since the
complex projective space has one-dimensional cohomology in each even degree
and no odd cohomology, we have 
\begin{equation*}
        \begin{array}{l}
                p+q< n\Longrightarrow h^{p,q}(V)=\delta _{p,q}\text{ \ and \ }h_{\text{%
                                prim}}^{p,q}(V)=\left \{ 
                \begin{array}{ll}
                        1, & \text{if }(p,q)=(0,0) \\ 
                        0, & \text{otherwise}%
                \end{array}%
                \right.  \\ 
                p+q=n\Longrightarrow h^{p,q}(V)=h_{\text{prim}}^{p,q}(V)+\delta _{p,q}.%
        \end{array}%
\end{equation*}%
So $h^{p,q}(V)\notin \{0,1\} \Longrightarrow p+q=n$.  Next, we define the \textit{middle primitive Hodge numbers}  of $V$,  \begin{equation}
	h_{\text{prim,mid}}^{p,q}(V):=\left \{ 
	\begin{array}{ll}
		h_{\text{prim}}^{p,q}(V), & p+q=n, \medskip \\ 
		0, & p+q\neq n,%
	\end{array}%
	\right.  \label{MIDPRHODGE}
\end{equation}
as well as the \textit{primitive
        Hodge polynomial} of $V$ (with non-negative coefficients) as follows: 
\begin{equation}
  P_{\text{prim}}(V;u,v):=\sum_{p=0}^{n}h_{\text{prim}%
  }^{p,n-p}(V)u^{p}v^{n-p}=\sum_{\substack{ p,q\geq 0 \\ p+q=n}}h_{\text{prim}%
  }^{p,q}(V)u^{p}v^{q}=\sum_{p,q\geq 0}h_{\text{prim,mid}}^{p,q}(V)u^{p}v^{q}.  \label{PRIMITIVHPOL}
\end{equation}
For the Hodge--Deligne polynomial, every primitive middle term has total
degree $n,$ and therefore 
\begin{equation}
        \fbox{$%
                \begin{array}{ccc}
                        & E(V;u,v)=\sum \limits_{j=0}^{n}(uv)^{j}+(-1)^{n}P_{\text{prim}}(V;u,v). & 
                \end{array}%
                $}  \label{eq:Epoly}
\end{equation}
\textit{Conventions}: (i) To make the preceding formulas valid  also in dimension $n=0,
$ whenever $V$ consists of $r\geq 1$ points, we define 
\begin{equation*}
	H_{\text{prim}}^{0}(V,\mathbb{C}):=H_{\text{prim}}^{0,0}(V,\mathbb{C}):=%
	\widetilde{H}^{0}(V,\mathbb{C})\cong \mathbb{C}^{r-1},
\end{equation*}%
with $h_{\text{prim}}^{0,0}(V)=h_{\text{prim,mid}}^{0,0}(V)=r-1.$ By this definition \eqref{eq:Epoly} remains valid
because%
\begin{equation*}
	E(\mathbb{P}^{0};u,v)=1,\text{ \ }E(V;u,v)=r,\text{ \ with \ }P_{\text{prim}%
	}(V;u,v)=r-1.
\end{equation*}%
(ii) For the empty complete intersection, which formally corresponds here to expected dimension \(-1\), we set $P_{\text{prim}}(\varnothing; u,v):=0$. 
\subsection{Hirzebruch's \(\chi_y\)-genus for smooth complete intersections\label{SUNSA3}}
For $V=V_{(d_{1},...,d_{c})}\subset \PP^{\nu }$ a smooth complete
intersection of dimension $n=\nu - c \ge 0$  we set 
\begin{equation*}
        \chi _{y}(V):=\sum_{p=0}^{n}\chi \left( V,\Omega _{V}^{p}\right) y^{p},\text{
                where }\chi \left( V,\Omega _{V}^{p}\right)
        =\sum_{q=0}^{n}(-1)^{q}h^{p,q}(V)=(-1)^{p}+(-1)^{n-p}h_{\text{prim}%
        }^{p,n-p}(V).
\end{equation*}%
Thus, 
\begin{equation}
        h_{\text{prim}}^{p,n-p}(V)=(-1)^{n-p}\chi \left( V,\Omega _{V}^{p}\right)
        -(-1)^{n}.  \label{hPRCHI}
\end{equation}%
By the Hirzebruch$-$Riemann$-$Roch theorem, 
\begin{equation*}
        \chi _{y}(V)=\text{ Res}_{z=0}\frac{z^{-\left( \nu +1\right) }dz}{(1+yz)(1-z)%
        }\prod \limits_{j=1}^{c}\frac{(1+yz)^{d_{j}}-(1-z)^{d_{j}}}{%
                (1+yz)^{d_{j}}+y(1-z)^{d_{j}}}.
\end{equation*}%
(See \cite[\S 2]{Hir54} or \cite[Thm. 22.1.1, p. 159]{Hir66}.) Evaluating
this residue, we obtain $\chi \left( V,\Omega _{V}^{p}\right) $ as the coefficient of $y^{p}$ in   $\chi _{y}(V)$, and then, via \eqref%
{hPRCHI} the primitive middle Hodge numbers of $V$: 
\begin{equation}
        \fbox{$h_{\text{prim}}^{p,n-p}(V)=\Bigg(\sum \limits_{\substack{ t,\,r_{1},\dots
                                ,r_{c}\geq 0  \\ t+r_{1}+\cdots +r_{c}\leq p}}\; \sum \limits_{J\subseteq
                        \{1,\dots ,c\}}(-1)^{\,t+\#J}\binom{\nu +1}{t}\binom{\,t+\sum%
                        \nolimits_{j=1}^{c}d_{j}r_{j}+\sum \nolimits_{j\notin J}d_{j}-1\,}{\, \nu }\Bigg)-(-1)^{n}.$%
        }  \label{hpqprimGEN}
\end{equation}
\vspace{0.15cm}
\subsection{Primitive Hodge numbers and $E$-polynomials of $Z_{m,d}$ and $Z^{\prime}_{m,d}$}~\par  
\noindent (i) Applying \eqref{hpqprimGEN} for the complete intersection $       Z_{m,d}:=V(g_0,..,g_m)\subset\PP^d$ of $m+1$ hypersurfaces, each of which has degree $d-1$, with $\text{dim}(Z_{m,d})=k=d-m-1$ (see (\ref{eq:CZ}) and (\ref{eq:DIASTA})) we get
\begin{equation}
        \fbox{$%
                \begin{array}{r}
                        h_{\text{prim}}^{\,p,k-p}(Z_{m,d})=\sum \limits_{i=0}^{p}\sum%
                        \limits_{j=1}^{m+1}(-1)^{m+1+i+j}\binom{d+1}{i}\binom{p-i+m+1}{m+1-j}\binom{%
                                p-i+j-1}{j-1} \medskip \\ 
                        \times \binom{(d-1)(p+j)-(d-2)i-1}{d}%
                \end{array}%
                $}  \label{hpqprimZMD}
\end{equation}%
for all $p\in\{0,...,k\}$ and \eqref{eq:Epoly}  gives
\begin{equation}
        \boxed{\begin{aligned}
                        E(      Z_{m,d};u,v)
                        ={}&
                        \sum_{t=0}^{k}(uv)^t
                        +
                        \sum_{r=0}^{k}
                        \sum_{i=0}^{r}
                        \sum_{j=1}^{m+1}
                        (-1)^{d+i+j}
                        \binom{d+1}{i}
                        \binom{r-i+m+1}{m+1-j} \medskip
                        \\
                        &\times
                        \binom{r-i+j-1}{j-1}
                        \binom{(d-1)(r+j)-(d-2)i-1}{d}
                        \,u^rv^{k-r}.
        \end{aligned}}
        \label{eq:E-I}
\end{equation}

\noindent (ii) Correspondingly, applying \eqref{hpqprimGEN} for the complete intersection $  Z'_{m,d}:=V(g_0,\ldots,g_m,h)\subset\PP^d$ of dimension $       l=d-m-2=k-1$ with $\text{deg}(h)=d$ (see \eqref{eq:CZ} and \eqref{eq:DIASTA}) we obtain
\begin{equation}
        \fbox{$%
                \begin{array}{r}
                        h_{\text{prim}}^{\,p,l-p}(Z_{m,d}^{\prime
                        })=\Bigg(\sum \limits_{i=0}^{p}(-1)^{m+i}\binom{d+1}{i}\sum \limits_{j=0}^{p-i}%
                        \sum \limits_{s=0}^{p-i-j}\binom{s+m}{m}\sum \limits_{t=0}^{m+1}(-1)^{t}%
                        \binom{m+1}{t} \medskip \\ 
                        \times \left[ \binom{i+(d-1)(s+t+j)+j-1}{d}-\binom{i+(d-1)(s+t+j+1)+j}{d}\right]\Bigg)-(-1)^{l} %
                \end{array}%
                $}  \label{hpqprimZMDpr}
\end{equation}
for all $p\in\{0,...,l\}$ and \eqref{eq:Epoly}  gives
\begin{equation}
        \boxed{\begin{aligned}
                        E(Z'_{m,d};u,v)
                        ={}&
                      \Bigg(
                        \sum_{r=0}^{l}
                        \sum_{a=0}^{r}
                        \sum_{b=0}^{r-a}
                        \sum_{i=0}^{r-a-b}
                        \sum_{j=0}^{m+1}
                        (-1)^{d+a+j}
                        \binom{d+1}{a}
                        \binom{i+m}{m}
                        \binom{m+1}{j}
                        \medskip \\
                        &\quad \hspace{1,5cm} \times
                        \left[
                        \tbinom{a+(d-1)(i+j+b)+b-1}{d}
                        -
                        \tbinom{a+(d-1)(i+j+b+1)+b}{d}
                        \right]
                        u^rv^{l-r}\Bigg)
                          \medskip \\
                        & \quad \hspace{6,5cm} + \sum_{t=0}^{l}(uv)^t
                        -  \sum_{r=0}^{l}u^{l-r}v^r
        \end{aligned}}
        \label{E-II}
\end{equation}

\end{document}